\documentclass[12pt]{amsart}

\usepackage[dvipsnames]{xcolor}
\usepackage{bbm}
\usepackage{graphicx}
\usepackage{amsmath}
\usepackage{color}
\usepackage{caption}
\usepackage{subcaption}
\usepackage{amsfonts,amssymb,amscd}
\usepackage{comment}

\calclayout

\newtheoremstyle{remboldstyle}
  {}{}{\itshape}{}{\bfseries}{.}{.5em}{{\thmname{#1 }}{\thmnumber{#2}}{\thmnote{ (#3)}}}
\theoremstyle{remboldstyle}

\newtheorem{thm}{Theorem}[section]
\newtheorem{prop}[thm]{Proposition}
\newtheorem{lem}[thm]{Lemma}
\newtheorem{cor}[thm]{Corollary}

\theoremstyle{definition}
\newtheorem{definition}[thm]{Definition}

\newtheorem{rem}[thm]{Remark}
\newtheorem{notation}[thm]{Notation}

\newtheorem{thmx}{Theorem}

\numberwithin{equation}{section}

\def\Mod{{\rm{Mod}}}
\def\Bbb{\mathbb}
\def\reals{\Bbb R}
\def\complex{\Bbb C}
\def\disk{\Bbb D}
\def\circle{\Bbb T}
\def\naturals{\Bbb N}
\def\integers{\Bbb Z}

\newcommand{\C}{\mathbb{C}}  

\newcommand{\Chat}{\widehat{\mathbb{C}}}

\DeclareMathOperator{\diam}{diameter}

\DeclareMathOperator{\dist}{dist}
\DeclareMathOperator{\diameter}{diam} 

\begin{document}


\title[Equilateral Triangulations and Postcritical Dynamics]{Equilateral Triangulations and The Postcritical Dynamics of Meromorphic Functions}


\author{Christopher J. Bishop, Kirill Lazebnik, and Mariusz Urba\'nski}





\begin{abstract} We show that any dynamics on any planar set $S$ discrete in some domain $D$ can be realized by the postcritical dynamics of a function holomorphic in $D$, up to a small perturbation. A key step in the proof, and a result of independent interest, is that any planar domain $D$ can be equilaterally triangulated with triangles whose diameters $\rightarrow0$ (at any prescribed rate) near $\partial D$.



\end{abstract}


\maketitle

\footnotetext[1]{MSC 30D05 \and MSC 37F10 \and 30D30.}


\tableofcontents


\section{Introduction}
\label{introduction_section}



We begin by briefly introducing some conventions. In what follows, we will use the spherical metric $d$ to measure distance between two points on $\Chat$ (see Definition \ref{spherical_metric}). If $D\subset\Chat$ is a domain, we will say a set $S\subset D$ is \emph{discrete} in $D$ if $S$ has no accumulation points in $D$. We define the \emph{singular values} of a holomorphic function $f: D \rightarrow \Chat$ to be the set $S(f)$ of critical values and \emph{asympotic values} of $f$. A point $w\in\Chat$ is an asymptotic value of $f: D \rightarrow \Chat$ if there exists a curve \begin{equation}\nonumber \gamma:[0,\infty)\rightarrow D \textrm{ with } \gamma(t)\xrightarrow{t\rightarrow\infty}\partial D \textrm{ and } f\circ\gamma(t)\xrightarrow{t\rightarrow\infty} w. \end{equation} The \emph{postsingular set} of $f$ is defined by  \begin{equation}\nonumber P(f):=\left\{ f^n(w) : w \in S(f) \textrm{ and } n\geq0 \right\}. \end{equation}

In the study of the dynamics of a holomorphic function $f:D \rightarrow\Chat$, a fundamental role is played by the sets $S(f)$, $P(f)$, and the behavior of $f$ restricted to $P(f)$. For instance, in the most well-studied cases $D=\mathbb{C}$, $\Chat$, the boundary of any Siegel disc of $f$ is contained in $\overline{P(f)}$, and much more generally, any component in the Fatou set of $f$ always necessitates a certain behavior for the orbit of a singular value of $f$ (see Section 4.3 of \cite{MR1216719} for $D=\mathbb{C}$, and \cite{MilnorCDBook} for $D=\Chat$). Thus, the following question arises: which dynamics on which sets $S\subset D$ can be realized by the postsingular dynamics of a holomorphic function $f: D \rightarrow \Chat$? Our first result (Theorem \ref{main_thm} below) says that as long as $S\subset D$ is discrete, \emph{any} dynamics on $S$ can be realized, up to a small perturbation. Before stating this result more precisely, we need:






\begin{definition} Let $\varepsilon>0$ and $X$, $Y\subset\Chat$. We say a homeomorphism $\phi: X\rightarrow Y$ is an \emph{$\varepsilon$-homeomorphism} if $\sup_{z\in X}d(\phi(z),z)<\varepsilon$. If a conjugacy $\phi$ between two dynamical systems is an $\varepsilon$-homeomorphism, we say $\phi$ is an \emph{$\varepsilon$-conjugacy}.



\end{definition}







\begin{thmx}\label{main_thm}\emph{ Let $D\subseteq\widehat{\mathbb{C}}$ be a domain, $S\subset D$ a discrete set with $|S|\geq3$, $h: S \rightarrow S$ a map, and $\varepsilon>0$. Then there exists an $\varepsilon$-homeomorphism $\phi:\Chat\rightarrow\Chat$ and a holomorphic map $f: \phi(D) \rightarrow \widehat{\mathbb{C}}$ with no asymptotic values such that $P(f)\subset \phi(D)$ and $f|_{P(f)}: {P(f)}\rightarrow P(f)$ is $\varepsilon$-conjugate to $h: S\rightarrow S$. }
\end{thmx}

As will be shown in Section \ref{proof_of_main_thm}, the $\varepsilon$-conjugacy between $f: P(f) \rightarrow P(f)$ and $h: S\rightarrow S$ is a composition of $\phi$ with a bijection of $S$ onto a perturbation of $S$.  When $D=\Chat$, Theorem \ref{main_thm} is very similar to Theorem 1.3 of \cite{MR4099617} (in \cite{MR4099617} the $\varepsilon$-conjugacy may be taken $=\phi$). When $D=\C$, Theorem \ref{main_thm} is very similar to Theorem 1 of \cite{MR4023391} (the difference being that functions in \cite{MR4023391} have asymptotic values and there the conjugacy $P(f)\mapsto S$ may be taken tangent to the identity at $\infty$). The main technique in \cite{MR4099617} is iteration in Teichm\"uller space, whereas in \cite{MR4023391} it is quasiconformal folding. The present manuscript provides a new approach that works simultaneously in both the settings $D=\Chat$, $\C$, as well as in much more general settings. We remark that our techniques do not answer whether for particular $S$ and $h: S\rightarrow S$ one can take $P(f)=S$ and $f|_{P(f)}=h$ (see Question 1.2 of \cite{MR4099617}). Related questions were also studied in \cite{MR1857667}, \cite{MR4216362}. We also remark that since the function $f$ of Theorem \ref{main_thm} has no asymptotic values, the postsingular set $P(f)$ coincides with the postcritical set of $f$. 



The proof of Theorem \ref{main_thm} proceeds by quasiconformally deforming a certain \emph{Belyi function} on $D$: a holomorphic map $g: D\rightarrow\Chat$ branching only over the three values $\pm1, \infty$. Given the existence of $g$, the main tools in the proof of Theorem \ref{main_thm} are the Measurable Riemann Mapping Theorem and an improvement of a fixpoint technique first introduced in \cite{MR4023391} (see also \cite{MR4008367}, \cite{MR4273483}). The existence of a Belyi function on $D$, on the other hand, will follow from the existence of a particular \emph{equilateral triangulation} of the domain $D$: a topological triangulation of $D$ with the property that for any two adjacent triangles $T, T'$, there is an anti-conformal reflection map $T\rightarrow T'$ which fixes pointwise the common edge (see Definitions \ref{equil_triang_defn}, \ref{equil_triang_def}). Indeed, given an equilateral triangulation $\mathcal{T}$ of $D$, after subdividing the equilateral triangulation if necessary (see Remark \ref{subdivision_remark}), a conformal map of a triangle $T\in\mathcal{T}$ to $\mathbb{H}$ (with the vertices of the triangle mapping to $\pm1$, $\infty$) may be extended to a Belyi function on $D$ by the Schwarz reflection principle. The connection between equilateral triangulations and Belyi functions was first described in \cite{MR988486}. The existence of the desired equilateral triangulation of $D$ will follow from the following Theorem, where we recall that the degree of a vertex $v$ in a triangulation $\mathcal{T}$ is defined as the number of edges in $\mathcal{T}$ having $v$ as a vertex:


\begin{thmx}\label{theorem_B}\emph{Let $D \subset \widehat \complex$
be a domain. Suppose $\eta: [0,\infty) \rightarrow [0,\infty)$ is  continuous, 
strictly increasing, and $\eta(0) =0$. 
Then there exists an equilateral
triangulation ${\mathcal T}$ of
$D$  so that for every $z \in D$ and 
every triangle $T \in {\mathcal T}$ 
containing $z$ we have 
\begin{eqnarray}\label{eta bound}
\diam(T) \leq \eta(d(z,\partial D)).
\end{eqnarray} 
Moreover, the degree of any vertex $v$  is bounded, independently of 
$v$, $D$ and $\eta$. }
\end{thmx}

The existence of an equilateral triangulation of $D$ is already implied by the recent result of \cite{2021arXiv210316702B}: that any non-compact Riemann surface can be equilaterally triangulated. In order to prove Theorem \ref{main_thm}, however, we will need to prove that the triangulation can also be taken to satisfy the condition (\ref{eta bound}). We remark that diameter in (\ref{eta bound}) refers to spherical diameter.

Theorem \ref{theorem_B} is a key step in the proof of Theorem \ref{main_thm}, but it is also of independent interest. As already partially alluded to, by \cite{MR988486} a Riemann surface $X$ has an equilateral triangulation if and only if it has a Belyi function $g: X \rightarrow \Chat$, in which case $g^{-1}([-1,1])$ is a so-called \emph{dessin d'enfant} on $X$. There is an extensive literature on dessins d'enfants (see \cite{MR2036721} for an overview), and of recent interest is the question of which geometries on a given Riemann surface a dessin may achieve. For instance, \cite{MR3232011} shows that unicellular dessins d'enfants are dense in all planar continua. Condition (\ref{eta bound}) is equivalent to a certain geometry for the corresponding dessin, and it is likely the techniques used in proving (\ref{eta bound})  will be of use in the question of attainable geometries for a dessin d'enfant on a given Riemann surface. 

We now briefly outline the paper. In Section \ref{outline} we will sketch the proofs of Theorems \ref{main_thm}, \ref{theorem_B}. In Sections \ref{moving}-\ref{proof_of_main_thm}, we prove Theorem \ref{main_thm} by first assuming Theorem \ref{theorem_B}, and in Sections \ref{conformal_grid_annuli}-\ref{Triangulating_domains} we prove Theorem \ref{theorem_B}. Sections \ref{conformal_grid_annuli}-\ref{Triangulating_domains} may be read independently of Sections \ref{moving}-\ref{proof_of_main_thm}. We will give a more detailed outline of the paper after sketching the main proofs in Section \ref{outline}.





\vspace{2mm}

\noindent \emph{Acknowledgements.} The authors would like to thank the anonymous referee for their suggestions which led to an improved version of the manuscript.

\section{Sketch of the Proofs}\label{outline}



In this Section, we sketch the proofs of Theorems \ref{main_thm}, \ref{theorem_B}. We begin with Theorem \ref{main_thm}, where the main ideas are already present in the case $D=\Chat$, and we discuss this case first.

Consider a sequence of equilateral triangulations $\mathcal{T}_n$ of $\Chat$ satisfying \begin{equation}\label{shrinking_diameters} \sup_{T\in\mathcal{T}_n} \diam(T) \xrightarrow{n\rightarrow\infty}0. \end{equation} The existence of $\mathcal{T}_n$ is trivial: see for instance Figure \ref{fig:subdivisionpicture.png}. As described above, any triangle $T\in\mathcal{T}_n$ and any vertex-preserving conformal map $T\mapsto \mathbb{H}(-1,1,\infty)$ (vertex-preserving means the three vertices of $T$ map to $\pm1$, $\infty$ under the conformal map) defines a holomorphic map $g: \Chat\rightarrow\Chat$.


\begin{figure}[ht!]
{\includegraphics[width=1\textwidth]{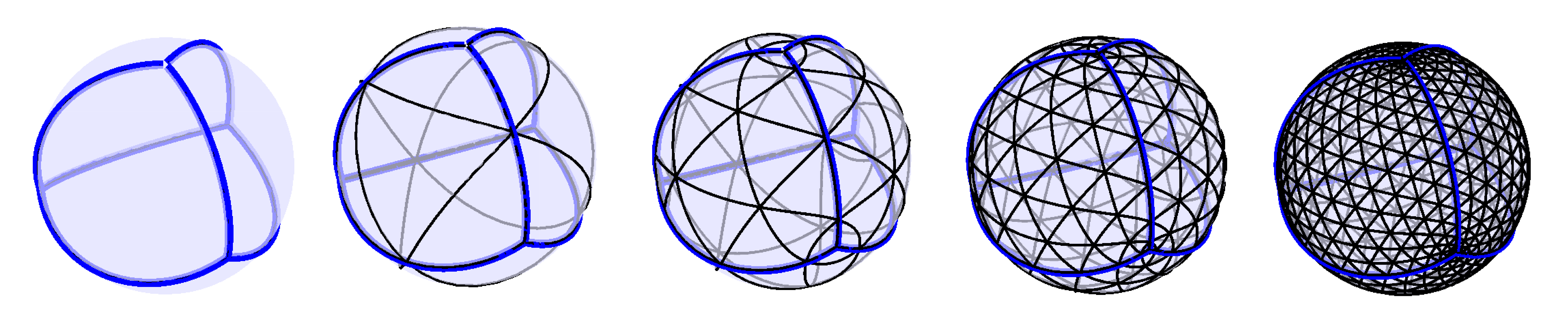}}
\caption{Illustrated is a sequence of triangulations $\mathcal{T}_n$ of $\Chat$. $\mathcal{T}_0$ is the tetrahedral subdivision of $\Chat$, and $\mathcal{T}_n$ is obtained from $\mathcal{T}_{n-1}$ by connecting the centers of each edge in each triangle in $\mathcal{T}_{n-1}$. }
\label{fig:subdivisionpicture.png}      
\end{figure}

The critical points of $g$ are precisely the vertices in the triangulation $\mathcal{T}_n$, and the critical values of $g$ are $\pm1, \infty$. For any vertex $v\in\mathcal{T}_n$, let $\mathcal{T}_{\{v\}}$  denote the union of triangles in $\mathcal{T}_n$ which have $v$ as a vertex. We can change the definition of $g|_{\mathcal{T}_{\{v\}}}$ to a map $\tilde{g}|_{\mathcal{T}_{\{v\}}}$  by post-composing $g|_{\mathcal{T}_{\{v\}}}$ with a quasiconformal map of $\Chat$ which perturbs the critical value $g(v)\in\{\pm1,\infty\}$ to a parameter $\tilde{g}(v)\in\Chat$, in such a way that $\tilde{g}|_{\partial\mathcal{T}_{\{v\}}}=g|_{\partial\mathcal{T}_{\{v\}}}$. Doing so over a sparse subset of vertices in $\mathcal{T}_n$, we call this new quasiregular map $\tilde{g}: \Chat\rightarrow\Chat$. 

Given a discrete (finite) $S\subset\Chat$ and a map $h:S \rightarrow S$, we choose a vertex $v_s\in\mathcal{T}_n$ nearby each $s\in S$, and consider the family of mappings $\tilde{g}$ determined by a choice of $(\tilde{g}(v_s))_{s\in S}$. Each such choice $(\tilde{g}(v_s))_{s\in S}$ determines a holomorphic map $f:=\tilde{g}\circ\phi^{-1}$, where $\phi$ is a quasiconformal mapping obtained from the Measurable Riemann Mapping theorem. In order to obtain the conjugacy between $f: P(f) \rightarrow P(f)$ and $h:S\rightarrow S$, the main idea (see also Figure \ref{fig:fixpoint_picture}) is to justify that we can choose $(\tilde{g}(v_s))_{s\in S}$ so that \begin{equation}\label{want_this_intro} \tilde{g}(v_s)=\phi(v_{h(s)}), \textrm{ for all } s\in S.\end{equation} Indeed, suppose we have the relation (\ref{want_this_intro}), and assume for simplicity that $h$ is onto. Then we would have \begin{equation} P(f)=\tilde{g}\left((v_s)_{s\in S}\right)=\phi\left((v_{h(s)})_{s\in S}\right) = \phi\left((v_s)_{s\in S}\right), \end{equation} and the desired conjugacy between $f: P(f) \rightarrow P(f)$ and $h:S\rightarrow S$ would be defined by $\phi(v_{s})\mapsto s$, since: \begin{equation} f\left( \phi(v_{s}) \right)  = \tilde{g}\circ\phi^{-1} \circ \phi(v_{s})  = \tilde{g}(v_{s}) = \phi(v_{h(s)}).  \end{equation} That we can choose each $\tilde{g}(v_s)$ so that (\ref{want_this_intro}) holds is non-trivial. The dilatation of $\tilde{g}$, and hence the mapping $\phi$, depends on the parameter $\tilde{g}(v_s)$ in a non-explicit manner (by solution of the Beltrami equation). Nevertheless, we can show the desired choice of $\tilde{g}(v_s)$ exists by application of a fixpoint theorem, where the variable is the set of parameters $\tilde{g}(v_s)$ and the output is the set of points $\phi(v_{h(s)})$. Moreover, if $n$ is large, the triangulation $\mathcal{T}_n$ is fine by (\ref{shrinking_diameters}) and the dilatation of $\phi$ small, so that $\phi(v_{s})\approx v_s\approx s$, and hence the conjugacy is close to the identity. Much of the technical work in Sections \ref{moving}-\ref{proof_of_main_thm} is in setting up the parameters $n$, $\tilde{g}(v_s)$ so that the hypotheses of an appropriate fixpoint theorem hold.

The crucial property of the domain $D=\Chat$ that was used in the above sketch was the existence of the equilateral triangulations $\mathcal{T}_n$ of $D$. While this property is trivial in the cases $D=\Chat$, $D=\mathbb{C}$ and it is well known in many other cases, it is non-trivial in the general setting. This is the content of Theorem \ref{theorem_B}. The main idea of the proof of Theorem \ref{theorem_B} is as follows. Assume $\infty\in D$, and let $K:=\partial D$. We consider sets $\Gamma_k$ which are contours surrounding $K$ (see Figure \ref{Whitney}). The desired triangulation $\mathcal{T}$ is produced by an inductive procedure. 
Roughly speaking, at the $k^{\textrm{th}}$ step we define the triangulation ${\mathcal T}_k$ to
equal the previous triangulation ${\mathcal T}_{k-1}$
outside $\Gamma_{k}$ 
and equal a Euclidean equilateral triangulation 
inside $\Gamma_{k}$. However, these two triangulations 
need to be merged in a very thin neighborhood
of $\Gamma_{k}$ (with a non-equilateral triangulation) and a quasiconformal correction 
is then applied to make the merged  triangulation 
equilateral. The dilatation of the correction map is 
supported in a thin neighborhood of $\Gamma_{k}$, and 
is chosen so thin that so the correction map is 
close to the identity. The desired triangulation $\mathcal{T}$ is then the limit of the triangulations $\mathcal{T}_k$ as $k\rightarrow\infty$.


We now give a detailed outline of the rest of the paper. In Section \ref{moving} we describe how we will change the map $g|_{\mathcal{T}_{\{v\}}}$ to the map $\tilde{g}|_{\mathcal{T}_{\{v\}}}$, introducing the parameters $\tilde{g}(v_s)$. In Section \ref{equil_triang_sec}, we deduce from Theorem \ref{theorem_B} the only result (Theorem \ref{desired_triangulation}) about equilateral triangulations we will need in order to prove Theorem \ref{main_thm}. In Section \ref{abasefamily}, we introduce the family of mappings amongst which we will find our desired fixpoint, and prove some estimates about this family. In Sections \ref{finding} and \ref{proof_of_main_thm}, we conclude the proof of Theorem \ref{main_thm} (modulo the proof of Theorem \ref{theorem_B}) by applying a fixpoint theorem. In Section \ref{conformal_grid_annuli} we introduce the regions in which we will merge equilateral triangulations, and we triangulate them in Section \ref{Triangulating_strips}. In Section \ref{Triangulating_domains} we construct the contours $\Gamma_k$ surrounding $K$ and prove Theorem \ref{theorem_B}.

\section{Moving a Critical Value}
\label{moving}

In this short Section we set up the framework we will need in order to be able to perturb the critical values of the function $g$ described in the Introduction. First we recall the definition of the spherical metric (see Section I.1.1 of \cite{MR0344463}):

\begin{definition}\label{spherical_metric} Two finite points $z_1$, $z_2\in\mathbb{C}$ have spherical distance \begin{equation} d(z_1, z_2):=\arctan\left|\frac{z_1-z_2}{1+\overline{z_1}z_2}\right| \textrm{ where } 0 \leq d(z_1, z_2) \leq \pi/2\textrm{,} \end{equation} and $d(z_1,\infty)=\arctan\left|1/z_1\right|$.
\end{definition}

\noindent We will use the basic theory of quasiconformal mappings throughout this paper, for which we refer the reader to the standard references \cite{Ahlfors-QCbook} and \cite{MR0344463}. 

\begin{notation} If $\phi$ is a quasiconformal mapping, we will denote its Beltrami coefficient $\phi_{\overline{z}}/\phi_z$ by $\mu(\phi)$. 
\end{notation}

\begin{definition}\label{pertubation_maps} For $w\in\{\pm1,\infty\}$, let $I_w$ be the subarc of $\hat{\mathbb{R}}:=\mathbb{R}\cup\{\infty\}$ with endpoints in $\{\pm1,\infty\}\setminus\{w\}$ which does not pass through $w$ (so for instance, $I_{-1}=(1,\infty)$). Given $w\in\{\pm1,\infty\}$ and $\zeta\in\Chat$ satisfying $d(\zeta,I_w)\geq\pi/12$, we will define a quasiconformal map $\phi^\zeta_{w}: \Chat\rightarrow\Chat$ as follows. Let \begin{enumerate} \item $\phi^\zeta_w: B(w,\pi/24)\rightarrow B(\zeta,\pi/24)$ be the restriction to $B(w,\pi/24)$ of an isometry of $\hat{\mathbb{C}}$ mapping $w$ to $\zeta$, \item $\phi^\zeta_{w}(z)=z$ for $z \in I_w$, \item $\phi^\zeta_{w}(z)$ is a smooth interpolation between (1) and (2) on $\hat{\mathbb{C}}\setminus\left(I_w\cup B(w,\pi/24)\right)$, and \item $\mu(\phi^\zeta_{w})$ varies smoothly with respect to $\zeta$.
\end{enumerate}
\end{definition}

The mapping $\phi^\zeta_w$ of Definition \ref{pertubation_maps} exists, and we make note of the following:

\begin{rem}\label{rem_perturbation_covering} The constant $\pi/12$ in Definition \ref{pertubation_maps} is chosen because $\pi/6=2\pi/12$, and \begin{equation}\label{perturbation_covering}\bigcup_{w\in\{\pm1,\infty\}} \left\{\zeta\in\Chat: d(\zeta,I_w)\geq\pi/6\right\} = \Chat. \end{equation} This fact will be important in the proof of Theorem \ref{main_thm}. 
\end{rem}

\begin{prop}\label{dilatation_constant} There exists $0<k_0<1$ such that for any $\zeta\in\Chat$, there is $w\in\{\pm1, \infty\}$ such that $||\mu(\phi^\zeta_w)||_{L^\infty(\Chat)}<k_0$.
\end{prop}

\begin{proof} Fix $w\in\{\pm1,\infty\}$. There exists $\zeta\in\Chat$ satisfying $d(\zeta,I_w)\geq\pi/12$. Fix such a $\zeta$. We have that $\phi^\zeta_w$ is a quasiconformal mapping, and moreover $\mu(\phi^\zeta_w)$ varies continuously with respect to $\zeta$ by (4) of Definition \ref{pertubation_maps}. Thus, as $||\mu(\phi^\zeta_w)||_{L^\infty(\Chat)}<1$ for each $\zeta$ satisfying $d(\zeta,I_w)\geq\pi/12$, we have that \[ \sup_{\zeta\in\{\zeta \hspace{.5mm}: \hspace{.5mm} d(\zeta,I_w)\geq\pi/12\}}||\mu(\phi^\zeta_w)||_{L^\infty(\Chat)} < 1. \] The result now follows from (\ref{perturbation_covering}).

\end{proof}



\section{Equilateral Triangulations}\label{equil_triang_sec}

In this Section, we will deduce from Theorem \ref{theorem_B} the only result (Theorem \ref{desired_triangulation}) we will need about equilateral triangulations in order to prove Theorem \ref{main_thm}. First we fix our definitions and some notation: 

\begin{definition}\label{equil_triang_defn} Let $D\subset\Chat$ be a domain. A \emph{triangulation} of $D$ is a countable and locally finite collection of closed topological triangles in $D$ that cover $D$, such that two triangles intersect only in a full edge or at a vertex.
\end{definition}

\begin{definition}\label{equil_triang_def} Let $D\subset\Chat$ be a domain, and $\mathcal{T}$ a triangulation of $D$. We say $\mathcal{T}$ is an \emph{equilateral triangulation} if for any two triangles $T$, $T'$ in $\mathcal{T}$ which share an edge $e$, there is an anti-conformal map of $T$ onto $T'$ which fixes pointwise the edge $e$ and sends the vertex opposite $e$ in $T$ to the vertex opposite $e$ in $T'$.
\end{definition}

\begin{rem} Definition \ref{equil_triang_def} readily generalizes to a definition of equilateral triangulations for Riemann surfaces. If a Riemann surface $S$ is built by gluing together Euclidean equilateral triangles, then the corresponding triangulation of $S$ satisfies Definition \ref{equil_triang_def}. The converse is also true. In other words, if a triangulation of a Riemann surface $S$ satisfies Definition \ref{equil_triang_def}, then $S$ can be constructed by gluing together Euclidean equilateral triangles (finitely many triangles if $S$ is compact, countably many if $S$ is non-compact). This justifies the terminology ``equilateral triangulation'' of Definition \ref{equil_triang_def}. See \cite{MR988486} or \cite{2021arXiv210316702B} for details.
\end{rem}

\begin{definition}\label{requestedadjacentdefn} Let $\mathcal{T}$ be a triangulation of a domain $D$. We say that two vertices $v$, $w \in \mathcal{T}$ are \emph{adjacent} if they are connected by an edge in $\mathcal{T}$. Otherwise we say $v$, $w$ are \emph{non-adjacent}. Similarly, two triangles in $\mathcal{T}$ are said to be \emph{adjacent} if they share a common edge, otherwise they are said to be \emph{non-adjacent} (in particular two triangles which intersect only at a vertex are non-adjacent).
\end{definition}

\begin{notation} Given a subset $\mathcal{V}$ of vertices in a triangulation $\mathcal{T}$, we will denote by $\mathcal{T}_\mathcal{V}$ the union of those triangles in $\mathcal{T}$ with at least one vertex in $\mathcal{V}$. In what follows, area will refer to spherical area, and $\diam$ to spherical diameter.  
\end{notation}


\begin{thm}\label{desired_triangulation} Let $D\subset\Chat$ be a domain and $S$ a discrete set in $D$. Then there exists a sequence of equilateral triangulations $\{\mathcal{T}_n\}_{n=1}^\infty$ of $D$ and a collection of pairwise non-adjacent triangles $\{T_s^n\}_{s\in S} \subset \mathcal{T}_n$ for each $n$ satisfying: \begin{enumerate}  \item $s\in T_s^n$ for all $s \in S$ and $n\in\mathbb{N}$,  \item  For any choice of vertices $v_s^n \in T_s^n$ we have: \begin{equation} \sum_{s \in S}\emph{area}\left( \mathcal{T}_{\{v_s^n\}} \right) \xrightarrow{n\rightarrow\infty} 0, \textrm{ and } \end{equation}  \item \begin{equation}\label{small_triangles} \sup_{s\in S}\diam(T_s^n)\xrightarrow{n\rightarrow\infty}0. \end{equation} \end{enumerate}

\end{thm}

\vspace{2mm}


\noindent \emph{Proof of Theorem \ref{desired_triangulation} assuming Theorem \ref{theorem_B}.}  Label the elements of $S$ as $\{s_k\}_{k=1}^{|S|}$ so that \begin{equation}\label{listing_s} d(s_1,\partial D)\geq d(s_2,\partial D)\geq d(s_3,\partial D)\geq... \end{equation} 

We will build a sequence of continuous, strictly increasing functions $(\eta_n)_{n=1}^\infty: [0,\infty) \rightarrow [0,\infty)$ satisfying $\eta_n(0)=0$ to which we will apply Theorem \ref{theorem_B}. We start with $\eta_1$. Let $c_k:=d(s_k, \partial D)$, where we note that $c_k\rightarrow0$ if $S$ is infinite. Since $S$ is discrete in $D$, we have that
\begin{equation} I_k:=\{l \in\mathbb{N}: c_l=c_k\} \end{equation}
is finite for every $k$. Hence we may define $\eta_1$ to be positive and strictly increasing in a small neighborhood of each $c_k$ so that 
\begin{equation}\label{1/2} \eta_1(c_k) <  \frac{ d(s, S\setminus\{s\})}{2} \textrm{ for all } s\in I_k, \textrm{ and } \end{equation}
\begin{equation}\label{1/2s} \eta_1(c_k+\eta_1(c_k)) < \frac{1}{2^k}. \end{equation} Finish the definition of $\eta_1$ by setting $\eta_1(0)=0$ and interpolating on the rest of $[0,\infty)$. We let \begin{equation}\label{eta_defn} \eta_n:=\eta_1/n. \end{equation}


Theorem \ref{theorem_B} applied to $(\eta_n)_{n=1}^\infty$ yields a sequence of equilateral triangulations $\{\mathcal{T}_n\}_{n=1}^\infty$ of $D$. We define the collection $\{T_s^n\}_{s\in S} \subset \mathcal{T}_n$ by setting $T_s^n$ to be any triangle in $\mathcal{T}_n$ containing $s$. By (\ref{eta bound}), (\ref{1/2}) and (\ref{eta_defn}), we have that if $s, s' \in S$ with $s\not=s'$, then $T_s^n$, $T_{s'}^n$ are non-adjacent for any $n$. Let $v_s^n$ be any choice of vertex in $T_s^n$ for each $s\in S$ and $n\in\mathbb{N}$. Since $v_s^n \in T_s^n$, we have by Theorem \ref{theorem_B} that \[ d(v_s^n, \partial D) <  d(v_s^n, s) + d(s, \partial D) \leq \eta_n(c_k) + c_k. \] Thus, again by Theorem \ref{theorem_B}, we have that if $T$ is a triangle with the vertex $v_s^n$, then \[ \diam(T) \leq \eta_n(c_k+\eta_n(c_k)). \] Recalling that the maximal degree of a vertex in any of the triangulations $\mathcal{T}_n$ is bounded by a universal constant (call it $d$) by Theorem \ref{theorem_B}, it follows from (\ref{1/2s}) and (\ref{eta_defn}) that: \begin{equation}  \sum_{s \in S}\textrm{area}\left( \mathcal{T}_{\{v_s^n\}} \right) \leq d\cdot\sum_{k=1}^{|S|} \left[\eta_n(c_k+\eta_n(c_k))\right]^2 \leq \frac{d}{n^2}\cdot\sum_{k=1}^{|S|} \left[\eta_1(c_k+\eta_1(c_k))\right]^2 \xrightarrow{n\rightarrow\infty}0.  \end{equation} Thus Property (2) in the conclusion of the Theorem is proven. Property (1) holds by definition of $T_s^n$, and Property (3) follows from Property (1), Theorem \ref{theorem_B}, and the observation that \[ \sup_{k\in\mathbb{N}} \eta_n(c_k) \xrightarrow{n\rightarrow\infty} 0. \]



\qed



\section{A Base Family of Mappings}
\label{abasefamily}

Having proven Theorem \ref{desired_triangulation}, we now have the holomorphic function $g: D\rightarrow\Chat$ described in the Introduction (see Definition \ref{g_n_defn} below). In this Section, we introduce a family of quasiregular perturbations of $g$ by moving critical values of $g$ using the results of Section \ref{moving}. The application we have in mind is roughly to prove Theorem \ref{main_thm} by finding a fixpoint in this family, and so we will need to establish certain technical estimates about this family which roughly correspond to verifying the hypotheses of an appropriate fixpoint theorem.




\begin{rem} Throughout Section \ref{abasefamily} we will fix a domain $D\subset\Chat$, a discrete set $S\subset D$, and equilateral triangulations $\mathcal{T}_n$ of $D$ as given in Theorem \ref{desired_triangulation}. 
\end{rem}


\begin{rem}\label{subdivision_remark} A triangulation is called $3$-colourable if its vertices may be coloured with three distinct colours in such a way that adjacent vertices have different colours. Any triangulation can be subdivided into a $3$-colourable triangulation by barycentric subdivision (see Figure \ref{fig:barycentric_subdivision}). Since barycentric subdivision preserves the properties of Theorem \ref{desired_triangulation}, we may assume that the triangulations $\mathcal{T}_n$ are $3$-colourable, and and each vertex has an even degree. This allows us to define the following (see also Remark 2.8 of \cite{2021arXiv210316702B}):
\end{rem}

\begin{figure}[ht!]
{\includegraphics[width=.66\textwidth]{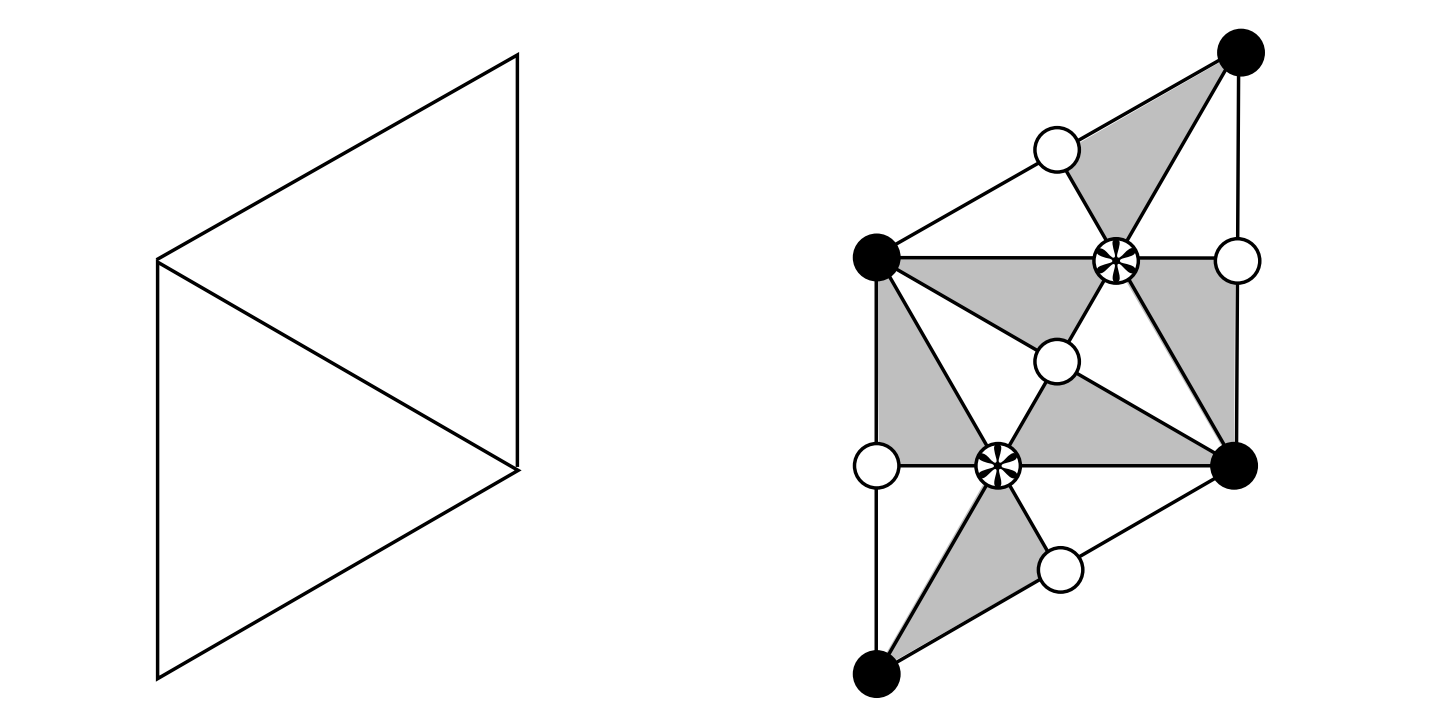}}
\caption{Illustrated is the process of barycentric subdivision. This figure is borrowed from \cite{2021arXiv210316702B}.}
\label{fig:barycentric_subdivision}
\end{figure}

\begin{definition}\label{g_n_defn} We will define a sequence of holomorphic maps $g_n: D\rightarrow\Chat$ as follows. For any $n$, fix a triangle $T\in\mathcal{T}_n$, and let $g_n: T \rightarrow \mathbb{H}(-1,1,\infty)$ be a conformal map such that the vertices of $T$ map to $\pm1,\infty$.  The definition of $g_n$ on $D$ is then obtained by application of the Schwarz reflection principle. 
\end{definition}

\begin{prop}\label{crit_points_of_g} The critical points of $g_n$ are precisely the vertices of the triangles in $\mathcal{T}_n$. The only critical values of $g_n$ are $\pm1,\infty$.
\end{prop}

\begin{proof} The maps $g_n$ are locally univalent except at the vertices of triangles in $\mathcal{T}_n$. At a vertex $v$ in $\mathcal{T}_n$, the map $g_n$ is locally $m:1$ where $m$ is such that $2m$ edges of the triangulation $\mathcal{T}_n$ meet at $v$. The last statement follows since each vertex is sent to one of $\pm1$, $\infty$ by $g_n$.
\end{proof}


\begin{prop}\label{quasiregular_map} Let $n>0$, let $\mathcal{V}$ be a subset of pairwise non-adjacent vertices in $\mathcal{T}_n$, and suppose we have a mapping $\tilde{h}: \mathcal{V}\rightarrow\Chat$. If $d(\tilde{h}(v),I_{g_n(v)})\geq\pi/12$ for each $v\in\mathcal{V}$, then for $k_0$ as in Proposition \ref{dilatation_constant}, there exists a quasiregular mapping $\tilde{g}_n: D\rightarrow\Chat$ such that: 
\begin{enumerate} \item $\tilde{g}_n(v)=\tilde{h}(v)$ for all $v\in\mathcal{V}$,
 \item $\tilde{g}_n\equiv g_n$ on $(\cup\mathcal{T}_n)\setminus \mathcal{T}_\mathcal{V}$ and hence $\mu(\tilde{g}_n)$ is supported on $\mathcal{T}_\mathcal{V}$, and  
  \item $||\mu(\tilde{g}_n)||_{L_\infty(D)}<k_0$. 
  \end{enumerate}
\end{prop}

\begin{proof} We will abbreviate $g=g_n$, and assume as in the statement of the Proposition that $d(\tilde{h}(v),I_{g(v)})\geq\pi/12$ for each $v\in\mathcal{V}$. Thus, the quasiconformal map $\phi^{\tilde{h}(v)}_{g(v)}$ of Definition \ref{pertubation_maps} satisfies: \begin{equation}\label{first_prop_phi} \phi^{\tilde{h}(v)}_{g(v)}(g(v))=\tilde{h}(v) \textrm{ (by (1) of Definition \ref{pertubation_maps}), }\end{equation} \noindent and  \begin{equation}\label{second_prop_phi}  \left|\left|\mu\left(\phi^{\tilde{h}(v)}_{g(v)}\right)\right|\right|_{L^\infty(\Chat)}<k_0 \textrm{ (by Proposition \ref{dilatation_constant}) }\end{equation} for all $v\in\mathcal{V}$. For any $v\in\mathcal{V}$, we define \begin{equation}\label{defn_pert} \tilde{g}_n:=\phi^{\tilde{h}(v)}_{g(v)}\circ g \textrm{ in } \mathcal{T}_{\{v\}},\end{equation} and \begin{equation}\label{defn_no_pert} \tilde{g}_n:=g  \textrm{ in } (\cup\mathcal{T}_n)\setminus \mathcal{T}_\mathcal{V}. \end{equation} Note that (\ref{defn_pert}) is well-defined since we have assumed no two vertices in $\mathcal{V}$ are adjacent. Moreover, since the boundary of $\mathcal{T}_{\{v\}}$ is mapped to $I_{g(v)}$, (2) of Definition \ref{pertubation_maps} implies that the Definitions (\ref{defn_pert}) and (\ref{defn_no_pert}) coincide along $\partial \mathcal{T}_\mathcal{V}$. Thus, by removability of analytic arcs for quasiregular mappings (see for instance Theorem I.8.3 of \cite{MR0344463}),  (\ref{defn_pert}) and (\ref{defn_no_pert}) define a quasiregular mapping on $\Chat$. Properties (1)-(3) in the statement of the Proposition now follow from (\ref{first_prop_phi})-(\ref{defn_no_pert}).
\end{proof}


\begin{rem}\label{permissibility_definition} Following the hypotheses of Proposition \ref{quasiregular_map}, we will call $n$, $\mathcal{V}$, $\tilde{h}$ \emph{permissible} if $d(\tilde{h}(v),I_{g_n(v)})\geq\pi/12$ for each $v\in\mathcal{V}$. We use the notation $\tilde{h}$ since this mapping will later be chosen to approximate the mapping $h$ of Theorem \ref{main_thm}. The mapping $\tilde{g}_n$ is completely determined by a choice of permissible $n$, $\mathcal{V}$, $\tilde{h}$, so that a more precise (but more cumbersome) notation for $\tilde{g}_n$ would be $\tilde{g}_{n, \mathcal{V}, \tilde{h}}$. Instead, we will usually omit all of these parameters and simply denote the mapping by $\tilde{g}$, with the dependence on $n$, $\mathcal{V}$, and $\tilde{h}$ understood.

\end{rem}

\begin{rem}\label{no_asymp_values_remark} We recall the definition of an \emph{asymptotic value}. A value $w\in\Chat$ is an asymptotic value of a holomorphic function $f: D \rightarrow \Chat$ if there exists a curve $\gamma:[0,\infty)\rightarrow D$ with $\gamma(t)\rightarrow\partial D$ as $t\rightarrow\infty$ such that $f\circ\gamma(t)\rightarrow w$ as $t\rightarrow\infty$. As mentioned in the Introduction, the function $f$ of Theorem \ref{main_thm} has no asymptotic values, and hence the postcritical set and postsingular set of $f$ coincide. This will follow from the following Proposition (see also the proof of Theorem \ref{mainthmmod1}): \end{rem} 

\begin{prop}\label{no_asymp_values} Let $n$, $\mathcal{V}$, $\tilde{h}$ be permissible. Then the only branched values of $\tilde{g}$ are $\{\pm1,\infty\}\cup \tilde{h}(\mathcal{V})$. Moreover, if $\gamma:[0,\infty)\rightarrow D$ is a curve with $\gamma(t)\rightarrow\partial D$ as $t\rightarrow\infty$, then $\tilde{g}\circ\gamma(t)$ does not converge as $t\rightarrow\infty$.
\end{prop}

\begin{proof} By Proposition \ref{crit_points_of_g}, the only branched values of $g$ are $\pm1,\infty$, so it follows from (\ref{first_prop_phi}) and (\ref{defn_pert}) that the only branched values of $\tilde{g}$ are $\{\pm1,\infty\}\cup \tilde{h}(\mathcal{V})$. 

Let $\gamma:[0,\infty)\rightarrow D$ be a curve with $\gamma(t)\rightarrow\partial D$ as $t\rightarrow\infty$. Suppose by way of contradiction that there exists $w\in\Chat$ such that $\tilde{g}\circ\gamma(t)\rightarrow w$ as $t\rightarrow\infty$. By Definition \ref{equil_triang_defn} and (2) of Proposition \ref{quasiregular_map}, $\gamma([0,\infty))$ must cross infinitely many edges $e$ of the triangulation $\mathcal{T}_n$ such that $\tilde{g}(e)\subset\hat{\mathbb{R}}$. Thus we must have $w\in\hat{\mathbb{R}}$. On the other hand, consider any Jordan curve $\Gamma$ passing through $\pm1,\infty$ with $\Gamma\cap\hat{\mathbb{R}}=\{\pm1,\infty\}$. Then we similarly see $\gamma([0,\infty))$ must cross infinitely many edges of the triangulation $\tilde{g}^{-1}(\Gamma)$, and so $w\in\Gamma\cap\hat{\mathbb{R}}=\{\pm1,\infty\}$. But \begin{equation} \tilde{g}^{-1}\left( \bigcup_{w\in\{\pm1,\infty\}}B(w,\pi/12) \right) \end{equation} is a disconnected subset of $D$, and so there can not be $w\in\{\pm1,\infty\}$ such that $\tilde{g}(\gamma(t))\in B(w,\pi/12)$ for all sufficiently large $t$.

\end{proof}

\begin{thm}\label{vertex_choice} Let $h: S\rightarrow S$ and $\varepsilon>0$. Then for all sufficiently large $n$, there exists a set of pairwise non-adjacent vertices $\mathcal{V}_n\subset \mathcal{T}_n$ such that:
\begin{enumerate} \item There exists an $\varepsilon$-bijection $\psi_n: S\rightarrow\mathcal{V}_n$,
\item $\emph{area}(\bigcup_{s\in S}\mathcal{T}_{\{\psi_{n}(s)\}})\rightarrow0$ as $n\rightarrow\infty$, \item If $\tilde{h}: \mathcal{V}_n\rightarrow\Chat$ is such that $\sup_{v\in \mathcal{V}_n}d(\tilde{h}(v), h\circ\psi_n^{-1}(v))\leq\pi/12$, then $n$, $\mathcal{V}_n$, $\tilde{h}$ are permissible. \end{enumerate} \end{thm} 

\begin{proof} 

Let $h: S\rightarrow S$ and $\varepsilon>0$. Recall the triangles $\{T_s^n\}_{s\in S}$ of Theorem \ref{desired_triangulation}. By Theorem \ref{desired_triangulation}, there exists $N$ such that we have $T_s^n \subset B(s,\varepsilon)$ for all $n\geq N$ and $s \in S$. We henceforth assume $n\geq N$, and prove the conclusions of Theorem \ref{vertex_choice} hold for such $n$.

We first define $\mathcal{V}_n$ and the bijection $\psi_n: S\rightarrow\mathcal{V}_n$. Let $s\in S$. We will define $\psi_n(s)$ to be one of the three vertices of the triangle $T_s^n$: in order to determine which vertex, we first consider $h(s)$. By (\ref{perturbation_covering}), there is $w\in\{\pm1,\infty\}$ such that \begin{equation}\label{h(s)_relation} d(h(s), I_w)\geq\pi/6.    \end{equation} We define $\psi_n(s)$ to be the vertex $v$ of $T_s^n$ satisfying $g_n(v)=w$. This defines $\psi_n$ and $\mathcal{V}_n:=\psi_n(S)$, where we note $\psi_n$ is a bijection onto $\mathcal{V}_n$ since $T_s^n$, $T_{s'}^n$ are non-adjacent for distinct $s$, $s'$. That $\psi_n$ is an $\varepsilon$-bijection follows since $T_s^n \subset B(s,\varepsilon)$. Moreover, property (2) in the conclusion of Theorem \ref{vertex_choice} now also follows from property (2) of Theorem \ref{desired_triangulation}. 

We will now prove property (3). Let $s\in S$. Note that by our choice of $\psi_n(s)$ and the relation (\ref{h(s)_relation}) we have  that \[ d(h(s), I_{g_n\circ\psi_n(s)})\geq\pi/6.\] Thus, if $\zeta$ is such that $d(\zeta,h(s))\leq\pi/12$, we have \[d(\zeta, I_{g_n\circ\psi_n(s)})\geq\pi/12.\] Thus for any $\tilde{h}: \mathcal{V}_n\rightarrow\Chat$ such that \[\sup_{v\in \mathcal{V}_n}d(\tilde{h}(v), h\circ\psi_n^{-1}(v))\leq\pi/12,\] we have \[\inf_{v\in \mathcal{V}_n}d(\tilde{h}(v),I_{g_n(v)})\geq\pi/12.\] Thus as defined in Remark \ref{permissibility_definition}, we have that $n$, $\mathcal{V}_n$, $\tilde{h}$ are permissible.

\end{proof}


\begin{rem}\label{V_notation} The vertex set $\mathcal{V}_n$ in the conclusion of Theorem \ref{vertex_choice} is determined by a choice of $n$, $h$, $\varepsilon$. When we wish to emphasize this dependence, we will use the notation $\mathcal{V}(n, h, \varepsilon)$. We also remark that we will sometimes simply write $\psi$ in place of $\psi_n$ when $n$ is understood from the context.
\end{rem}

\begin{rem} Recall that the mapping $\tilde{g}$ is determined by permissible $n$, $\mathcal{V}$, $\tilde{h}$. In particular, the parameters $n$, $\mathcal{V}$, $\tilde{h}$ also determine (by way of the Measurable Riemann Mapping Theorem) a unique quasiconformal mapping $\phi: \Chat\rightarrow\Chat$ such that \begin{enumerate} \item $\tilde{g}\circ\phi^{-1}: \phi(D)\rightarrow\Chat$ is holomorphic, \item $\phi$ fixes each of $\pm1$, $\infty$, and \item $\mu(\phi)=0$ on $\Chat\setminus D$. \end{enumerate} As for $\tilde{g}$, we will omit the dependence of $\phi$ on the parameters $n$, $\mathcal{V}$, $\tilde{h}$ in our notation.
\end{rem}

\begin{prop}\label{phi_close_to_id} Let $h: S\rightarrow S$, and $\varepsilon>0$. For all sufficiently large $n$, we have that if $\tilde{h}$ is such that $n$, $\mathcal{V}(n, h, \varepsilon)$, $\tilde{h}$ are permissible, then \begin{equation}\label{close_to_id} \sup_{z\in\Chat}d(\phi(z),z)<\varepsilon. \end{equation}


\end{prop} 

\begin{proof} Let $h: S\rightarrow S$, and $\varepsilon>0$. Let $N$ be sufficiently large so that $\mathcal{V}(N, h, \varepsilon)$ is defined, let $n\geq N$, and let $\tilde{h}$ be such that $n$, $\mathcal{V}(n, h, \varepsilon)$, $\tilde{h}$ are permissible. Then \begin{equation} \textrm{supp}(\phi_{\overline{z}}) \subset \bigcup_{v\in\mathcal{V}(n, h, \varepsilon)} T_{\{v\}} = \bigcup_{s\in S} T_{\psi_n(s)} \end{equation}  Thus, by (2) of Theorem \ref{vertex_choice}, we have \begin{equation}  \textrm{area}(\textrm{supp}(\phi_{\overline{z}})) \xrightarrow{n\rightarrow\infty}0. \end{equation} Lastly, we recall that by (3) of Proposition \ref{quasiregular_map}, we have $||\mu(\phi)||_{L_\infty(\Chat)}<k_0<1$, in other words $\phi$ is $k_0$-quasiconformal with $k_0$ independent of $n$, $\mathcal{V}(n, h, \varepsilon)$, $\tilde{h}$. The result now follows from the fact that there exists $\delta>0$ such that if $\phi:\Chat\rightarrow\Chat$ is any normalized $k_0$-quasiconformal mapping with $\textrm{area}(\textrm{supp}(\phi_{\overline{z}}))<\delta$, then (\ref{close_to_id}) holds (see for instance Lemma 2.1 of \cite{MR3232011}).

\end{proof}

\section{Continuity of a Fixpoint Map}
\label{finding}

In Section \ref{proof_of_main_thm}, we will prove Theorem \ref{main_thm}. As already described in the Introduction, the main strategy is to describe the desired function in the conclusion of the theorem as the fixpoint of a particular mapping we call $\Upsilon$ (see Definition \ref{psi_definition_environment} and Figure \ref{fig:fixpoint_picture}). The estimates proven in Section \ref{abasefamily} will allow us to verify the appropriate continuity and contraction properties of $\Upsilon$ in order to apply a fixpoint theorem. Section \ref{finding} is dedicated to defining $\Upsilon$ and proving continuity. 

\begin{definition}\label{psi_definition_environment} Let $D$, $S$, $h$, $\varepsilon$ be as in Theorem \ref{main_thm} and let $n$ be sufficiently large so that $\mathcal{V}(n,h,\varepsilon/2)$ is defined (see Remark \ref{V_notation}). We will define a map  \begin{align}\label{fixpoint_map} \Upsilon: \prod_{t\in h(S)} \overline{B(t,\pi/12)} \rightarrow \prod_{t\in h(S)} \Chat  \end{align} as follows. Let \[ (\zeta_t)_{t\in h(S)} \in  \prod_{t\in h(S)} \overline{B(t,\pi/12)}. \] Define a mapping \[ \tilde{h}: \mathcal{V}(n,h,\varepsilon/2) \rightarrow \Chat \textrm{ by } \tilde{h}\circ\psi(s)=\zeta_{h(s)} \textrm{ for all } s \in S, \] where $\psi=\psi_n$ is the bijection of Theorem \ref{vertex_choice}. By (3) of Theorem \ref{vertex_choice}, the triple $n$, $\mathcal{V}(n,h,\varepsilon/2)$, $\tilde{h}$ is permissible, and hence determines the mappings $\tilde{g}$, $\phi$. We define: \begin{equation}\label{defn_of_psi} \Upsilon\left((\zeta_t)_{t\in h(S)}\right):=\left(\phi\circ\psi(t)\right)_{t\in h(S)}. \end{equation} 

\end{definition}

\begin{figure}[ht!]
{\includegraphics[width=1\textwidth]{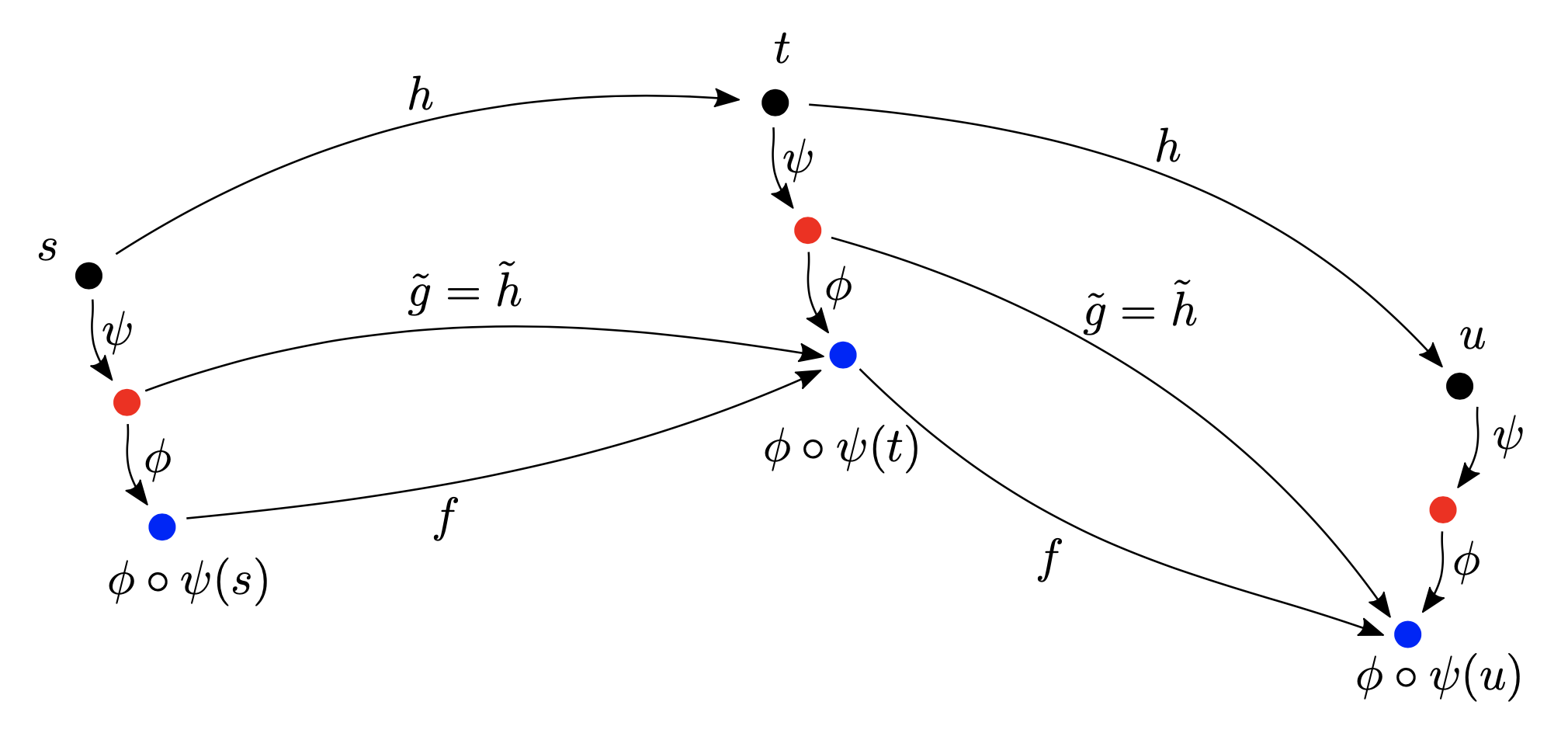}}
\caption{ Illustrated is the behavior of a fixpoint of the mapping $\Upsilon$. In black are points $s$, $t$, $u \in S$. In red are vertices of triangles containing $s$, $t$, $u$. In blue are the perturbations of these vertices under the correction mapping $\phi$. }
\label{fig:fixpoint_picture}      
\end{figure}



\begin{rem} We will always consider any product space $\prod_{i\in I}X_i$ to be endowed with the standard product topology. Recall that this topology is generated by subsets of the form  $\prod_{i\in I}U_i$ where each $U_i\subset X_i$ is open and $U_i=X_i$ except for finitely many $i$. With this topology, Tychonoff's Theorem says that any product of compact sets is compact. In particular, the domain of the mapping $\Upsilon$ is compact. 
\end{rem}

\begin{thm}\label{continuity_of_psi} The mapping $\Upsilon$ of Definition \ref{psi_definition_environment} is continuous.
\end{thm}

\begin{proof} Fix \[ (\zeta_t^0)_{t\in h(S)}=\zeta^0 \in  \prod_{t\in h(S)} \overline{B(t,\pi/12)}\textrm{ and } (\xi_t^0)_{t\in h(S)}:=\Upsilon(\zeta^0).\] Let $V\subset \prod_{t\in h(S)} \Chat $ be an open set containing $\Upsilon(\zeta^0)$. Since $V$ is open, there is an $\varepsilon'>0$ such that \[\prod_{t\in h(S)}B(\xi^0_t, \varepsilon')\subset V. \] Thus, in order to prove the Theorem, it suffices to show that there exists $\delta>0$ and a finite subset $\{t_1, ..., t_m\} \in h(S)$ such that if we define \begin{align}\label{defn_of_U_t} U_{t}:=B(\zeta_{t}^0, \delta) \textrm{ for } t\in\{t_1, ..., t_m\},  \hspace{6mm}   \\ \nonumber U_t:=\overline{B(t,\pi/12)} \textrm{ for } t\in h(S)\setminus\{t_1, ..., t_m\}, \end{align} then $U:=\prod_{t\in h(S)}U_t$  satisfies: \begin{equation}\label{want_to_show_this} \Upsilon(U)\subset \prod_{t\in h(S)}B(\xi^0_t, \varepsilon'). \end{equation} In fact, we will show something stronger than (\ref{want_to_show_this}). For \[ \zeta \in \prod_{t\in h(s)}\overline{B(t,\pi/12)},\] let $\phi^\zeta:\Chat\rightarrow\Chat$ denote the quasiconformal mapping of Definition \ref{psi_definition_environment}, and let $\phi_0:=\phi^{\zeta_0}$. We will show that there exists $\delta>0$ so that: \begin{equation}\label{we_want_to_show_this_final} \sup_{z\in\Chat}d(\phi^\zeta(z),\phi_0(z))<\varepsilon' \textrm{ for all } \zeta\in U. \end{equation}  


Recall the constant $k_0<1$ of Proposition \ref{dilatation_constant}. We will use the following two facts:  \begin{itemize} \item There exists $\delta'>0$ such that if $\phi:\Chat\rightarrow\Chat$ is any normalized $k_0$-quasiconformal mapping with $\textrm{area}(\textrm{supp}(\phi_{\overline{z}}))<\delta'$, then \begin{equation}\label{phi_close_to_id2} \sup_{z\in\Chat}d(\phi(z),z)<\varepsilon'/2. \end{equation} \vspace{1mm} \item[$(\ast\ast)$] There exists $\delta''>0$ such that if $\phi:\Chat\rightarrow\Chat$ is any normalized $\delta''$-quasiconformal mapping, then (\ref{phi_close_to_id2}) holds. \end{itemize} We will abbreviate $\mathcal{V}:=\mathcal{V}(n, h, \varepsilon/2)$.  Note that: \begin{equation} \textrm{supp}(\phi^\zeta_{\overline{z}}) \subset \bigcup_{v\in\mathcal{V}}T_v \textrm{ for all } \zeta\in  \prod_{t\in h(s)}\overline{B(t,\pi/12)}.\end{equation} Since \begin{equation} \sum_{v\in \mathcal{V}}\textrm{area}(T_v) < \textrm{area}(\Chat) <\infty, \end{equation} there exist $v_1$, ..., $v_m \in \mathcal{V}$ such that \begin{equation}\label{area_distortion} \sum_{v\in \mathcal{V}\setminus\{v_1, ..., v_m\}}\textrm{area}(T_v) < \delta'/C, \end{equation} where $C>0$ is such that any normalized $k_0$-quasiconformal mapping $\phi$ satisfies $\textrm{area}(\phi(E))\leq C\cdot\textrm{area}(E)$ for all measurable $E\subset\Chat$.  In (\ref{defn_of_U_t}), we let \begin{equation}\label{set_definition} \{t_1, ..., t_m\}:=\{ h\circ\psi(v_1), ...,  h\circ\psi(v_m)\}. \end{equation}Denote $A:=\cup_{1\leq i \leq m}T_{v_i}$, and for $\zeta\in U$, let $\phi_1^\zeta:\Chat\rightarrow\Chat$ denote the normalized integrating map for $\mathbbm{1}_{A}\cdot\mu(\phi^\zeta)$. By (4) of Definition \ref{pertubation_maps} and  (\ref{defn_of_U_t}), there exists $\delta''>0$ so that \begin{equation} ||\mu( \phi_1^\zeta\circ\phi_0^{-1})||_{L^\infty(A)}<\delta'' \textrm{ for } \zeta\in U. \end{equation} Let $\phi_2^\zeta: \Chat\rightarrow\Chat$ be such that $\phi_2^\zeta$ is conformal in $\Chat\setminus\phi_1^\zeta(D)$, and $\phi_2^\zeta\circ\phi_1^\zeta$ is the normalized integrating map for $\mu(\phi^\zeta)$, so that we have $\phi_2^\zeta\circ\phi_1^\zeta=\phi^\zeta$. Then \begin{equation}\textrm{supp}(\mu(\phi_2^\zeta))\subset\phi_1^\zeta\bigg(\bigcup_{\mathcal{V}\setminus\{v_1,...,v_n\}}T_v\bigg), \end{equation} and so by (\ref{area_distortion}), we have: \begin{equation} \textrm{area}(\textrm{supp}(\mu(\phi_2^\zeta)))<C\cdot\sum_{v\in \mathcal{V}\setminus\{v_1, ..., v_m\}}\textrm{area}(T_v)  \leq \delta'.\end{equation} Thus by combining ($\ast$) and ($\ast\ast$) we have that for $\zeta\in U$: \begin{align} \nonumber \sup_{z\in\Chat}d(\phi^\zeta_2\circ\phi_1^\zeta(z),\phi_0(z)) =  \sup_{z\in\Chat}d(\phi_2^\zeta\circ\phi_1^\zeta\circ\phi_0^{-1}(z),z) \hspace{20mm}\\ \leq \sup_{z\in\Chat}d(\phi_2^\zeta\circ\phi_1^\zeta\circ\phi_0^{-1}(z),\phi_1^\zeta\circ\phi_0^{-1}(z))+\sup_{z\in\Chat}d(\phi_1^\zeta\circ\phi_0^{-1}(z),z) <  \varepsilon'/2+ \varepsilon'/2=\varepsilon'. \nonumber \end{align} This is the relation (\ref{we_want_to_show_this_final}) which we needed to show. 

\end{proof}

\begin{rem} A map very similar to $\Upsilon$ was considered in \cite{MR4023391} (see Lemma 14 there), however there the proof of continuity was considerably simpler than in the present context. The added difficulty in the present setting is due to the fact that the map \[ \prod_{t\in h(S)} \overline{B(t,\pi/12)} \mapsto L^\infty(\Chat) \] (given by considering the Beltrami coefficient of the quasiregular map generated by any element in the domain) is not continuous, whereas in \cite{MR4023391} the domain of this map is different: it consists of a product of discs with radii $\rightarrow0$ and hence there the map into $L^\infty(\Chat)$ is continuous.
\end{rem}

\noindent We conclude Section \ref{finding} by recording the statement of the classical Schauder-Tychonoff fixpoint theorem (see for instance Theorem 5.28 of \cite{MR1157815}) which we will apply in the proof of Theorem \ref{main_thm}:

\begin{thm}\label{schauder_tychonoff} Let $V$ be a locally convex topological vector space. For any non-empty compact convex set $X$ in $V$, any continuous function $f: X \rightarrow X$ has a fixpoint.
\end{thm}

\section{Finding a Fixpoint}\label{proof_of_main_thm}

We now turn to the proof of Theorem \ref{main_thm}. It will be convenient to first prove a slightly modified version of the Theorem (see Theorem \ref{mainthmmod1} below), where we assume $\pm1$, $\infty \in h(S)$ and consider the map $h|_{h(S)}$ rather than $h$. We will also first assume the following condition holds:

\begin{definition}\label{normalizably_triangulable} Let $D\subseteq\widehat{\mathbb{C}}$ be a domain, $S\subset D$ a discrete set, $h: S \rightarrow S$ a map, and $\varepsilon>0$. We say $D$, $S$, $h$, $\varepsilon$ are \emph{normalizably triangulable} if there exist arbitrarily large $n$ such that the vertex set $\mathcal{V}=\mathcal{V}(n,h,\varepsilon/2)$ of Theorem \ref{vertex_choice} satisfies \begin{enumerate} \item $\pm1$, $\infty \in \mathcal{V}$, and \item $\psi(s)=s$ for $s\in\{\pm1,\infty\}$. \end{enumerate}
\end{definition}


\noindent As we will see, Theorem \ref{main_thm} will follow easily from the following Theorem:

\begin{thm}\label{mainthmmod1} Let $D\subseteq\widehat{\mathbb{C}}$ be a domain, $S\subset D$ a discrete set, $h: S \rightarrow S$ a map with $\pm1$, $\infty\in h(S)$, and $\varepsilon>0$. Assume $D$, $S$, $h$, $\varepsilon$ are normalizably triangulable. Then there exists an $\varepsilon$-homeomorphism $\phi:\Chat\rightarrow\Chat$ and a holomorphic map $f: \phi(D) \rightarrow \widehat{\mathbb{C}}$ with no asymptotic values such that $P(f)\subset \phi(D)$ and $f|_{P(f)}: {P(f)}\rightarrow P(f)$ is $\varepsilon$-conjugate to $h|_{h(S)}: h(S)\rightarrow h(S)$.
\end{thm}

\begin{proof}

\noindent We let $D$, $S$, $h$, $\varepsilon$ be as in the statement of Theorem \ref{mainthmmod1}. Fix $n>0$ sufficiently large so that the conclusions of Theorem \ref{vertex_choice} and Proposition \ref{phi_close_to_id} hold for $h: S\rightarrow S$ and $\varepsilon/2$, and so that the vertex set $\mathcal{V}:=\mathcal{V}(n,h,\varepsilon/2)$ is as in Definition \ref{normalizably_triangulable}.  By (3) of Theorem \ref{vertex_choice} and Proposition \ref{phi_close_to_id}, if $\tilde{h}: \mathcal{V}\rightarrow \Chat$ is any map such that \begin{equation}\label{second} \sup_{v\in \mathcal{V}}d(\tilde{h}(v), h\circ\psi^{-1}(v))\leq\pi/12, \end{equation} then $n$, $\mathcal{V}$, $\tilde{h}$ are permissible and \begin{equation}\label{first}\sup_{z\in\Chat}d(\phi(z),z)<\varepsilon/2. \end{equation} Thus, given \begin{equation}\label{again} (\zeta_t)_{t\in h(S)} \in \prod_{t\in h(S)}\overline{B(t,\pi/12)}, \end{equation} we define $\tilde{h}$ as in Definition \ref{psi_definition_environment} by \begin{equation} \tilde{h}\circ\psi(s):=\zeta_t \textrm{ for all } t\in h(S) \textrm{ and } s\in h^{-1}(t), \end{equation} which in turn defines the mappings  $\tilde{g}$, $\phi$, where $\phi$ satisfies (\ref{first}). 




Consider now the mapping $\Upsilon$ of Definition \ref{psi_definition_environment}. By (\ref{first}) and (1) of Theorem \ref{vertex_choice}, we have for any $(\zeta_t)_{t\in h(S)} \in \prod_{t\in h(S)}\overline{B(t,\pi/12)}$ that: \begin{equation}\label{phipsiid} d(\phi\circ\psi(t), t) \leq d(\phi\circ\psi(t), \psi(t))+d(\psi(t), t) < \varepsilon/2+\varepsilon/2=\varepsilon. \end{equation} Thus in fact $\Upsilon$ defines a map:  \begin{align}\label{fixpoint_map} \Upsilon: \prod_{t\in h(S)} \overline{B(t,\pi/12)}  \rightarrow \prod_{t\in h(S)} \overline{B(t,\varepsilon)}. \end{align} We claim that $\Upsilon$ has a fixpoint. Indeed, $\Upsilon$ is continuous by Proposition \ref{continuity_of_psi}, and the domain of $\Upsilon$ is compact and convex, so Theorem \ref{schauder_tychonoff} implies the existence of a fixpoint of $\Upsilon$. 

The fixpoint of $\Upsilon$ yields a choice of $\tilde{g}$, $\phi$ such that \begin{equation}\label{fixpoint_property} \tilde{h}\circ\psi(s)=\phi\circ\psi(t)\textrm{, for all } t\in h(S) \textrm{ and } s\in h^{-1}(t). \end{equation} Define the holomorphic map $f:=\tilde{g}\circ\phi^{-1}: \phi(D)\rightarrow\Chat$. We will show that $f$, $\phi$ satisfy the conclusions of the Theorem. We have already proven (see (\ref{first})) that $\phi$ is an $\varepsilon$-homeomorphism. We claim that $\{\pm1,\infty\}\subset \tilde{h}(\mathcal{V})$. Indeed, if $t\in\{\pm1,\infty\}$ and $s\in h^{-1}(t)$ (here we are using the assumption that $\pm1,\infty\in h(S)$), then by (\ref{fixpoint_property}) and (2) of Definition \ref{normalizably_triangulable} we have $\tilde{h}\circ\psi(s)=\phi\circ\psi(t)=\phi(t)=t$. Thus by Proposition \ref{no_asymp_values}, we conclude that $f$ has no asymptotic values and \begin{equation}\label{first_desired}  P(f) = \{\pm1,\infty\}\cup \tilde{h}(\mathcal{V}) =  \tilde{h}(\mathcal{V}).  \end{equation} Also, by (\ref{fixpoint_property}), we have: \begin{equation}\label{second_desired} \tilde{h}(\mathcal{V}) = \tilde{h}\circ\psi(S) = \phi\circ\psi(h(S)), \end{equation} and since $\psi(h(S))\subset D$ (since $\psi$ maps to vertices in a triangulation of $D$), we have $P(f)=\tilde{h}(\mathcal{V})\subset\phi(D)$. It remains to show that $f|_{P(f)}: P(f)\rightarrow P(f)$ and $h|_{h(S)}: h(S)\rightarrow h(S)$ are $\varepsilon$-conjugate. Indeed, we claim that $\phi\circ\psi: h(S) \rightarrow P(f)$ is the desired conjugacy. By (\ref{first_desired}) and (\ref{second_desired}) we have that $\phi\circ\psi: h(S) \rightarrow P(f)$ is onto and hence a bijection. By (\ref{phipsiid}), we have that $\phi\circ\psi: h(S) \rightarrow P(f)$ is an $\varepsilon$-bijection. Lastly, for all $t\in h(S)$: \begin{equation} f\circ\phi\circ\psi(t)=\tilde{g}\circ\psi(t)=\tilde{h}\circ\psi(t) = \phi\circ\psi\circ h(t), \end{equation} where the first $=$ is since $f:=\tilde{g}\circ\phi^{-1}$, the second $=$ is (1) of Proposition \ref{quasiregular_map}, and the last $=$ is by (\ref{fixpoint_property}).

\end{proof}

\noindent Now we remove the hypothesis of Definition \ref{normalizably_triangulable} from Theorem \ref{mainthmmod1}:

\begin{thm}\label{mainthmmod1'} Let $D\subseteq\widehat{\mathbb{C}}$ be a domain, $S\subset D$ a discrete set, $h: S \rightarrow S$ a map with $\pm1$, $\infty\in h(S)$, and $\varepsilon>0$. Then there exists an $\varepsilon$-homeomorphism $\phi:\Chat\rightarrow\Chat$ and a holomorphic map $f: \phi(D) \rightarrow \widehat{\mathbb{C}}$ with no asymptotic values such that $P(f)\subset \phi(D)$ and $f|_{P(f)}: {P(f)}\rightarrow P(f)$ is $\varepsilon$-conjugate to $h|_{h(S)}: h(S)\rightarrow h(S)$.
\end{thm}

\begin{proof}
We let $D$, $S$, $h$, $\varepsilon$ be as in the statement of Theorem \ref{mainthmmod1'}. Let $\varepsilon'>0$, and recall the bijection $\psi=\psi_{n, h, \varepsilon'}: \mathcal{V}(n, h, \varepsilon') \rightarrow S$ of Theorem \ref{vertex_choice}. Define a M\"obius transformation $M=M_n$ by \begin{equation}\label{containsnormalized} M\circ\psi_{n, h, \varepsilon'}(s)=s \textrm{ for } s\in\{\pm1,\infty\}. \end{equation} Then, by fixing $\varepsilon'$ sufficiently small, we have that $M$ is an $\varepsilon/2$-homeomorphism for all sufficiently large $n$. We define \begin{equation} S':=M(S\setminus\{\pm1,\infty\})\cup\{\pm1, \infty\}. \end{equation} We define $h': S' \rightarrow S'$ by a simple adjustment of the definition of $h$: \begin{equation}\label{h'definition} h'(s):= \begin{cases} 
	M\circ h\circ M^{-1}(s)  \hspace{4mm} \textrm{ if } s, h(s)\not\in\{\pm1,\infty\} \\
	h(s) \hspace{25mm} \textrm{ if } s, h(s)\in\{\pm1,\infty\} \\
        M\circ h(s) \hspace{17mm} \textrm{ if }  s \in \{\pm1, \infty\}, h(s)\not\in\{\pm1,\infty\} \\
        h\circ M^{-1}(s) \hspace{13mm} \textrm{ if }  s \not\in \{\pm1, \infty\}, h(s)\in\{\pm1,\infty\}
   \end{cases} \end{equation} For $n>0$, let $\mathcal{T}_n$ denote the triangulation of $D$ of Theorem \ref{vertex_choice}. Note that $M(\mathcal{T}_n)$ is a triangulation of $M(D)$, and moreover by (\ref{containsnormalized}) the vertex set $M(\mathcal{V}(n,h,\varepsilon/2))\subset M(\mathcal{T}_n)$ contains $\pm1,\infty$. Thus Theorem \ref{mainthmmod1} applies to $M(D)$, $S'$, $h'$, $\varepsilon/2$ to yield an $\varepsilon/2$-homeomorphism $\tilde{\phi}:\Chat\rightarrow\Chat$ and a holomorphic map $f: \tilde{\phi}\circ M(D)\rightarrow\Chat$ with no asymptotic values such that $P(f)\subset \tilde{\phi}\circ M(D)$ and $f: P(f)\rightarrow P(f)$ is $\varepsilon/2$-conjugate to $h'|_{h'(S')}: h'(S') \rightarrow h'(S')$. We claim that $\phi:=\tilde{\phi}\circ M$ and $f$ satisfy the conclusions of Theorem \ref{mainthmmod1'}.
   
Indeed, since $M$ is an $\varepsilon/2$-homeomorphism, it follows that $\phi=\tilde{\phi}\circ M$ is an $\varepsilon$-homeomorphism. We have already justified that $P(f)\subset \tilde{\phi}\circ M(D)$. Lastly, by Definition (\ref{h'definition}), $h'|_{h'(S')}: h'(S') \rightarrow h'(S')$ is $\varepsilon/2$-conjugate to $h|_{h(S)}: h(S) \rightarrow h(S)$, and so $f|_{P(f)}: P(f)\rightarrow P(f)$ is $\varepsilon$-conjugate to  $h|_{h(S)}: h(S) \rightarrow h(S)$.
   
   
   

\end{proof}

\noindent Next we remove the assumption that $\pm1$, $\infty \in h(S)$.

\begin{thm}\label{mainthmmod2} Let $D\subseteq\widehat{\mathbb{C}}$ be a domain, $S\subset D$ a discrete set with $|h(S)|\geq3$, $h: S \rightarrow S$ a map, and $\varepsilon>0$. Then there exists an $\varepsilon$-homeomorphism $\phi:\Chat\rightarrow\Chat$ and a holomorphic map $f: \phi(D) \rightarrow \widehat{\mathbb{C}}$ with no asymptotic values such that $P(f)\subset \phi(D)$ and $f|_{P(f)}: {P(f)}\rightarrow P(f)$ is $\varepsilon$-conjugate to $h|_{h(S)}: h(S)\rightarrow h(S)$.
\end{thm}

\begin{proof} We let $D$, $S$, $h$, $\varepsilon$ be as in the statement of Theorem \ref{mainthmmod2}. Let $M$ be a M\"obius transformation sending any three points of $h(S)$ to $\pm1$, $\infty$. Then applying Theorem \ref{mainthmmod1'} to $M(D)$, $M(S)$, $M\circ h\circ M^{-1}$, $\varepsilon(M)$ yields mappings we will denote by \begin{equation} \tilde{\phi}: \Chat\rightarrow\Chat \textrm{ and } \tilde{f}: \tilde{\phi}\circ M(D)\rightarrow\Chat. \end{equation} It is straightforward to then check that the functions $\phi:=M^{-1}\circ \tilde{\phi}\circ M$ and $f:=M^{-1}\circ \tilde{f}\circ M$ satisfy the conclusions of Theorem \ref{mainthmmod2} for aptly chosen $\varepsilon(M)$.
\end{proof}

\noindent In the case that $h$ is onto, Theorem \ref{mainthmmod2} is exactly Theorem \ref{main_thm}, and so all that remains is to consider the case that $h$ is not onto:

\vspace{2mm}

\noindent \emph{Proof of Theorem \ref{main_thm}.} We let $D$, $S$, $h$, $\varepsilon$ be as in the statement of Theorem \ref{main_thm}. We augment the set $S$ to a set $S'\supset S$ so that $S'$ is still discrete in $D$, and such that we can define a mapping $h': S' \rightarrow S'$ such that $h'(S')= S$ and $h'|_{S}=h$. Then since $h'|_{h'(S')}: h(S')\rightarrow h(S')$  is the same function as $h: S\rightarrow S$, applying Theorem \ref{mainthmmod2} to $D$, $S'$, $h'$, $\varepsilon$ yields the desired functions in the conclusion of Theorem \ref{main_thm}.

\section{Conformal Grid Annuli}
\label{conformal_grid_annuli}

In Sections \ref{conformal_grid_annuli}-\ref{Triangulating_domains}, we turn our attention to the proof of Theorem \ref{theorem_B}. As mentioned in the Introduction, Sections \ref{conformal_grid_annuli}-\ref{Triangulating_domains} may be read independently of Sections \ref{moving}-\ref{proof_of_main_thm}. We begin by studying the annuli in which we will interpolate between two different triangulations, as described in Section \ref{outline}. First we will need several definitions. 

\begin{rem} In Sections \ref{introduction_section}-\ref{proof_of_main_thm}, we used the spherical metric whenever measuring distance or diameter in the plane. In Sections \ref{conformal_grid_annuli}-\ref{Triangulating_domains} we will more often use Euclidean distance and Euclidean diameter, and we will denote these by $\dist$, $\diameter$, (respectively) to distinguish them from their spherical counterparts which we have been denoting by $d$, $\diam$. In fact, the distinction between the two metrics will not be crucial since the proof of Theorem \ref{theorem_B} only uses the Euclidean metric in a compact subset of $\mathbb{D}$ where it is Lipschitz-equivalent to the spherical metric. 
\end{rem}



\begin{definition} An {\it equilateral grid polygon} is a simple
closed polygon that lies on the edges of 
a Euclidean equilateral triangulation of the plane. 
An {\it equilateral grid annulus} is a topological 
annulus in $\reals^2$  so that the two boundary 
components are both equilateral grid polygons 
(on the same grid).
\end{definition} 

\begin{definition} Let $A$ be an equilateral grid polygon or annulus lying on the edges of a triangulation $\mathcal{T}$ with vertices $V$. The \emph{vertices} of $A$ are defined as $V\cap\partial A$. If a triangle $T\in\mathcal{T}$ has non-empty intersection with $\partial A$, we call $T$ a boundary triangle of $A$. 
If $A$ is an annulus, the \emph{thickness} of $A$ is defined as the minimum number of grid triangles needed to 
connect the two components of $\partial A$.
\end{definition}

\begin{notation} For any topological annulus $A$ in the plane, 
we let $\partial_o A$ and $\partial_i A$ denote 
the outer and inner connected  components of 
$\partial A$, in other words, $\partial_o A$ separates $A$
from $\infty$. 
\end{notation}
 
Recall that any  planar topological annulus  with 
non-degenerate boundary components 
can be conformally mapped to a round annulus 
of the form $B =\{ 1 < |z| < 1+\delta\}$, and
this map is unique up to rotation and inversion. We will be concerned primarily with the case where $\delta$ is small.

We wish to consider conformal images of 
equilateral grid annuli, but also a slightly 
more general class of annuli where each boundary component 
has a  one-sided neighborhood that is a conformal image 
of a equilateral grid annulus. More precisely, we define the following:

\begin{definition}\label{confgridanndefn} Let $A$ be a topological annulus so that both components of $\partial A$ are Jordan curves. We shall call $A$ a {\it conformal grid annulus}
if there exists a finite set $V\subset\partial A$ (called the {\it vertices} of $A$), two conformal maps $f_o, f_i$ on $A$ with the property that $f_o(A)$, $f_i(A)$ are topological annuli, and equilateral grid annuli $A_o, A_i$ so that for $k=o, i$: 
\begin{enumerate} 
\item $ A_k \subset f_k(A) $,
\item $\partial_k (f_k(A)) =\partial_k A_k$, 
\item $f_k(V \cap \partial_k A)$ equals the vertices on $\partial_k A_k$.
\end{enumerate}  

\end{definition}

If $f_o=f_i$ and $A_o=A_i$, conditions (1)-(3) in Definition \ref{confgridanndefn} just say that 
$A$ is the conformal image of a single 
equilateral grid annulus $A_o$ and the vertices of $A$ are the 
images of the vertices of $A_o$.

\begin{definition} The vertices of a conformal grid annulus $A$ naturally partition $\partial A$ into segments which we call the \emph{sub-arcs} of $\partial A$. We say two sub-arcs are \emph{adjacent} if they share a common endpoint. 
\end{definition}

\begin{definition} Given a conformal grid annulus $A$, we define 
$$ {\rm{inrad}}(A) = \sup_{z \in A} \dist(z, \partial A),$$
to be the \emph{in-radius} of $A$, and 
$$ {\rm{gap}}(A) = \sup \{ \diameter(\gamma) : \gamma\textrm{ is a sub-arc of } \partial A\}$$
to be the maximum (Euclidean) diameter of the sub-arcs of $\partial A$.
\end{definition}

 Later we will find 
triangulations of $A$ whose elements have diameters 
controlled by these quantities. 

Let notation be as in Definition \ref{confgridanndefn}. If $T$ is an outer boundary triangle of $A_o$, we will call
the topological triangle $f^{-1}_o(T)$ a boundary 
triangle of $A$. Similarly for $A_i$.
In our main application, the inner boundary of $A$ will be
an equilateral grid polygon and the $f_i$ will be the 
identity map. The associated boundary triangles of $A$ are 
then Euclidean equilateral. The outer boundary of $A$ will be the 
image of an equilateral grid polygon under a map $f_o^{-1}$ that 
extends conformally past $\partial_o A$. Thus the boundary 
triangles of $A$ along its outer boundary will be small, smooth 
perturbations of equilateral triangles. 

Below we shall use several 
standard properties of conformal modulus. This is a well known conformal invariant whose basic properties
are discussed in many  sources such as
\cite{Ahlfors-QCbook} or \cite{MR2150803}.
We briefly recall the basic definitions.
Suppose $\Gamma$ is a path family (a collection of locally rectifiable curves)
in a planar domain $\Omega$ and $\rho$ is a non-negative Borel
function  on $\Omega$. We say $\rho$ is admissible for $\Gamma$
(and write $\rho \in {\mathcal A}(\Gamma)$) if
$$ \ell(\Gamma) = \ell_\rho(\Gamma)
 = \inf_{\gamma \in \Gamma} \int_\gamma \rho ds \geq 1,$$
and define the modulus of $\Gamma$ as
$$ \Mod(\Gamma) = \inf_\rho  \int \rho^2 dxdy,$$
where the infimum is over all admissible $\rho$ for $\Gamma$.
We shall frequently use the following property of conformal modulus known as the extension rule: if $\Gamma, \Gamma'$
are path families so that every element $\gamma' 
\in \Gamma'$ equals or contains an 
element $\gamma \in \Gamma$ then $M(\Gamma) 
\leq M(\Gamma')$ (since if $\rho$ is admissible for $\Gamma$, it 
is also admissible for $\Gamma'$ so the infimum for 
$\Gamma$ is over a smaller set of metrics). We shall 
use the following basic facts later: the modulus of the path family connecting the 
two boundary components of $\{1<|z| < R\}$ is $2 \pi/
\log R$, and so the extension rule implies that any 
path family where every curve crosses such an 
annulus has modulus $\leq 2\pi/\log R$.

\begin{lem}  \label{comparable arcs}
Suppose $A$ is a conformal grid annulus and that 
there are at least four vertices on each 
component of $\partial A$. 
Suppose $f:A \to B= \{ z: 1< |z| < 1+\delta\}$ 
is a conformal map of $A$ onto a round annulus. 
This sends the sub-arcs on $\partial A$ to sub-arcs on 
$\partial B$.
Then there is an $M< \infty$, independent of $A$,
so that any two adjacent sub-arcs on $\partial B$ 
have lengths
comparable to within a factor $ M $, and  
every sub-arc in $B$ has length $\leq M \delta$. 
\end{lem} 

\begin{proof} 
Suppose  $J$  is a sub-arc of $\partial A$ and $I,K$ are the 
	two adjacent sub-arcs.
Let $\Gamma$ be the path family in $A$ that connects
$I$ to $K$.  If  $I,J,K$  are in the outer boundary 
of $A$ we let $f = f_o$ and  $A' = A_o$ and otherwise we set 
$f= f_i$ and $A' = A_i$. In either case we let
$I',J',K'$ be the corresponding line segments 
on the boundary of $A'$ and $\Gamma'$ the path 
family connecting $I'$ to $K'$ in $A'$.
Let $U$ be the union of all the boundary 
triangles of $A'$ that touch the boundary arc $\gamma' = 
I' \cup J'\cup K'$. Note that there are only 
finitely many shapes $\gamma'$ can have, and 
only finitely many shapes for $U$ (up to Euclidean similarity).

\begin{figure}[!htbp]
\centerline{
  \includegraphics[height=2in]{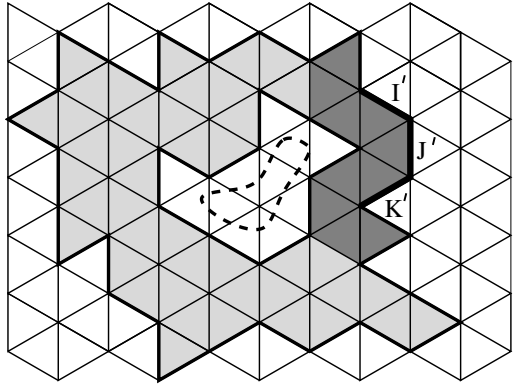}
}
\caption{ \label{GridPoly}
Here we assume that 
the outer boundary of $A$ maps to the outer 
boundary of a equilateral grid annulus $A'$ (shaded). 
The inner boundary of $f(A)$ (dashed) need not
coincide with the inner boundary of $A'$. Given 
three segments $I'$, $J'$ and $K'$ on the outer 
boundary of $A'$ we let $U$ be the union of 
all grid triangles in $A'$ that touch one of these
segments (darker shading). 
Since $I'$ and $K'$ don't touch each other 
and there are only finitely many possible 
shapes for $U$, 
the modulus of the path family connecting them in $A'$ is 
uniformly bounded.
}
\end{figure}

The path family $\Gamma'$  need 
not be the image of $\Gamma$ if $f(A) \ne A'$.
However, since $f$ is conformal and $A' \subset f(A)$
we have, by the extension rule 
that $M(\Gamma) \leq M(\Gamma')$. 
Again, $M(\Gamma')$ is one of a finite number of 
positive possibilities, so $M(\Gamma)$ is 
bounded uniformly from above.

We claim that $M(\Gamma)$ is 
also bounded uniformly from below. Let $\sigma$ be the union of the three line segments 
$I'$, $J'$, $K'$ and let $\Omega = \complex \setminus \sigma$.
By the conformal invariance of modulus together with the extension rule, $M(\Gamma)$ is bounded 
below by the modulus of the path family connecting $I'$ to 
$K'$ in $\Omega$, because $f(A) \subset \Omega$.
Again, this modulus is one of a finite number of 
positive possibilities, so $M(\Gamma)$ is 
bounded uniformly from below.

The modulus of the path family  in $A$ connecting $J$ 
to the component of $\partial A$ not containing $J$
is bounded above by the analogous path family for 
$J'$ in $A'$. This is bounded above by the 
modulus of the path family connecting $J'$ to 
$\partial U \setminus  \gamma'$. 
There are only a finite number of possible 
configurations of $U$ and $\gamma'$, and each 
gives a finite modulus, so the maximum
of these values is also bounded above, 
independent of $A$.

Thus for each arc $J$ on one component of 
$\partial B$, the path family $\tilde{\Gamma}$
connecting $J$ to the other component of $\partial B$ is bounded 
uniformly above. We claim that this implies $\textrm{length}(J)$ is $O(\delta)$  as $\delta\rightarrow0$. Indeed, the metric $\rho$ defined by setting $\rho(z)=1/\delta$ for $z\in B$ satisfying
\[ \textrm{arg}(z)\in\left(\min_{\zeta\in J}\arg(\zeta)-\delta, \max_{\zeta\in J}\arg(\zeta)+\delta\right)  \]
and $\rho(z)=0$ otherwise is admissible for $\tilde{\Gamma}$, and hence a calculation of $\int\rho^2dxdy$ together with the definition of modulus shows that $\textrm{length}(J)/\delta = O(\textrm{Mod}(\tilde{\Gamma}))$.

Similarly, the path family $\Gamma''$ in $B$ connecting arcs
$I,K$ that are both adjacent to an arc $J$ has modulus 
bounded uniformly above and below.  Recall that we have proven $\diameter(J)= O(\delta)$. Thus, if we suppose by way of contradiction that $\diameter(J)\not=O(\diameter(I))$ as $\delta\rightarrow0$, we would deduce that $M(\Gamma'')$ degenerates, a contradiction. We conclude that $\diameter(J) = O(\diameter(I))$ as $\delta\rightarrow0$. Since the roles 
of $I$ and $J$ may  be exchanged we deduce that
the two arcs have comparable lengths.
\end{proof} 

\begin{lem} \label{thickness}
For every $\varepsilon>0$, there is an $N\in\mathbb{N}$ so that if $A$
is a conformal grid annulus with $A_o$, $A_i$ each having 
thickness at least $N$, then in the conclusion of Lemma 
\ref{comparable arcs} each subarc on $\partial B$ has 
length at most  $\varepsilon \cdot \delta$. 
\end{lem} 

\begin{proof}
In this case, the path family connecting $J'$ to the opposite
boundary component must connect points in $J'$ to points outside 
a disk of radius $\simeq N \cdot \diameter(J')$ centered on $J'$. The 
extension rule and the modulus calculation for annuli 
then imply this 
path family has modulus tending to zero as $N$ increases to 
infinity. This implies the arc has small length
compared to the width  of $B$. 
\end{proof} 

For a rectifiable arc $\gamma$, we let $\ell(\gamma)$
denote the (Euclidean) length of $\gamma$. 
A homeomorphism $f: \gamma \to \sigma$ between 
rectifiable curves is said to multiply lengths 
if for any subarc $\gamma' \subset \gamma$ we 
have $\ell(f(\gamma')) = \ell(\gamma') \cdot 
\ell(\sigma)/\ell(\gamma)$.  

A rectifiable curve $\gamma$ is called an $M$-chord-arc if
for any two points $x,y \in \gamma$ the shortest sub-arc of $\gamma$ connecting 
$x$ and $y$ has length at most $M|x-y|$. A  map  $f$
is $L$-biLipschitz if 
$$ \frac 1L \leq \frac {|f(x)-f(y)|}{|x-y|} \leq L ,$$
for all $x,y$ in its domain, $x\not=y$. Bi-Lipschitz maps between
planar domains are automatically quasiconformal
with dilatation at most $K = L^2$. 
A closed curve is chord-arc if and only if it is the 
bi-Lipschitz image of a circle. A length multiplying 
map between two $M$-chord-arc curves is necessarily 
$M$-bi-Lipschitz, and moreover, an $L$-biLipschitz map 
between $M$-chord-arc curves has a
$K$-biLipschitz extension between the interiors, 
where $K$ only depends on $L$ and $M$. See e.g., 
\cite{MR639966} by Tukia or \cite{MR1370347} by MacManus. In the following proof, we use the notation $D(z,r)$ for the open (Euclidean) ball of radius $r>0$ centered at $z\in\mathbb{C}$.

\begin{lem} \label{length preserving}
In Lemma \ref{comparable arcs}, if $A$ is a conformal 
grid annulus and each boundary triangle $T$  of $A$ 
is an $L$-biLipschitz image of a 
Euclidean equilateral triangle, then 
there is a $K$-quasiconformal map $\psi: A \to B= \{ z: 1< |z| < 1+\delta\}$ so that for $f$ as in Lemma \ref{comparable arcs}, we have:
\begin{enumerate} 
\item  $\psi$ equals $f$ on $A$ minus the boundary triangles 
of $A$, 
\item $\psi$ equals $f$ on the boundary vertices of $A$, 
\item $\psi$ multiplies arclength on each boundary 
arc of $A$.
\item $K$ depends only on the biLipschitz constant $L$.
\end{enumerate} 
\end{lem} 

\begin{proof}
It is enough to consider the boundary corresponding 
to $A_o$; the argument for the inner boundary is the same. 

Let $f:A \to B$ be the conformal map of the conformal grid 
annulus $A$ to the round annulus $B$ given in Lemma 
\ref{comparable arcs}.
Consider a  boundary triangle $T'$ of the equilateral grid 
annulus $A_o$  and the corresponding boundary 
triangle $T = f_o^{-1}(T')$ of $A$. Then $g_T= f \circ f_o^{-1}$
is a conformal map of $T'$ into $B$. 
Recall that the  boundary of $A_o$ is a  grid polygon, so 
it has fixed 
side lengths (which we may assume are all unit length) 
and  every angle is in $\{\frac \pi 6, 
\frac \pi 3, \dots \frac {5\pi}6 \}$.  Thus at
each vertex $v$ of $\partial_o A_o$, the 
Schwarz reflection principle implies there is 
an $\alpha \in \{3, \frac 32, 1, \frac 34, \frac 35\}$ 
so that mapping $g_T((z-v)^\alpha))$ has a conformal 
extension to $D(v, \frac 12)$. 
This, together with the distortion theorem for conformal maps 
(e.g., Theorem I.4.5 of \cite{MR2150803}) 
implies that  each edge of $f(T)= g_T(T')$ 
is an analytic arc with uniform bounds, meeting the other
two at angles bounded uniformly away from zero 
(at interior verticies 
all angles are $\pi/6$ and at boundary vertices the 
angles are $\pi/k$ where $k$ vertices meet, and 
at most 5 triangles can meet a boundary vertex  
of a equilateral grid polygon). 
Thus the image topological triangle $f(T)$
is a chord-arc curve
with uniform bounds.
Define a map $\psi_T$  on the boundary 
of $T$ by making $\psi_T$ length multiplying on any 
edge lying on $\partial A$ and on  any edge in common with 
another boundary triangle, and let  $\psi_T = f$ on any 
other edges (necessarily an edge shared with a 
non-boundary triangle). This is a bi-Lipschitz 
map from $\partial T$ to $f(\partial T)$ between 
chord-arc curves and hence it has a bi-Lipschitz extension 
(which is also a quasiconformal extension)
between  the interiors with uniform bounds.
So if we replace $f$ in each boundary triangle $T$ by 
the map $\psi_T$, we get a  quasiconformal map $\psi:A \to B$ that 
satisfies all the desired properties. 
\end{proof} 

\begin{lem} \label{boundary triangles} 
Suppose $\Gamma$ is a equilateral grid polygon 
bounding a region $\Omega$ and $\gamma \subset 
\Omega$ is a equilateral grid polygon (on the same 
grid as $\Gamma$) so that the annulus between 
$\gamma$ and $\Gamma$ has thickness $N\geq10$. 
Let $\Omega' \subset \Omega$ be the region bounded 
by $\gamma$. 
Suppose $f$ is conformal on $\Omega$. Then there is 
$K$-quasiconformal map  $g$ on $\Omega'$ so that 
\begin{enumerate}
	\item $g=f$ off the triangles touching $\gamma$, 
	\item $g=f$ on the vertices of $\gamma$, 
	\item $g$ is length multiplying on the edges of $\gamma$.
	\item $K$ is absolute, and $K\rightarrow1$ as $N\rightarrow\infty$.
\end{enumerate}
\end{lem} 

\begin{proof}
For each boundary triangle $T$ of $\gamma$,  $f$ is 
conformal on a disk centered at the center of  $T$
with radius $\geq 4\cdot \diameter (T)$. Therefore the 
image $T' = f(T)$ consists of analytic arcs meeting 
at $60^\circ$. Thus for any subset of the three edges
of $T$ we can define a biLipschitz map 
$g : T \to T'$ that agrees with $f$ on this subset of 
edges, also agrees with $f$ at all three vertices, and 
is length multiplying on the remaining edges.
As above, this is a biLipschitz map 
between chord-arc curves so it has a 
biLipschitz (and hence quasiconformal) 
extension between the interiors, with constants 
that are uniformly bounded, say by $K$. 
 On any non-boundary triangle in $\Omega'$ we 
set $g =f$. For each boundary triangle  we take $g$
as above that is length multiplying 
on the edges of $T$ on $\gamma$ or shared with another 
boundary triangle, and  so that $g=f$  on edges of $T$ 
that are shared with a non-boundary triangle. 

If the thickness is very large, then $f(T)$ is 
close to an equilateral triangle, and it is clear
that the maps defined above can be taken close 
to isometries, in other words, the quasiconformal dilatation 
is close to $1$.
\end{proof}

\section{Triangulating Annuli} 
\label{Triangulating_strips}

In this Section, we triangulate the conformal grid annuli introduced in Section \ref{conformal_grid_annuli}. We do this by pulling back a triangulation of a conformally equivalent annulus by a certain quasiconformal mapping. We begin with a discussion of decomposition of domains into dyadic squares.

A dyadic interval $I\subset \reals$ is one of the 
form  $I = [j 2^{-n}, (j+1) 2^{-n}]$ 
for some integers $j,n$. 
A dyadic square in the plane is  a product 
of dyadic intervals of equal length, in other words, 
$Q=[j 2^{-n}, (j+1) 2^{-n}] 
 \times [k 2^{-n}, (k+1) 2^{-n}] $
for some integers $j,k,n$. We let 
$\ell(Q) = 2^{-n} = \diameter(Q)/\sqrt{2}$ denote the side 
length of $Q$. Two dyadic squares 
either have disjoint interiors or one is 
contained in the other one. 
Given a domain $D$, we can therefore take the set of maximal 
dyadic squares ${\mathcal W}=\{Q_j\}$ so  that 
$3Q_j \subset D$. Then 
\begin{eqnarray}\label{bounds}
\ell(Q_j)  \leq\dist(Q_j, \partial D) 
\leq 3 \sqrt{2} \ell(Q_j) . 
\end{eqnarray} 
This is an example of a Whitney decomposition of $D$. 
Note that if $Q$ and $Q'$ are adjacent squares in 
the Whitney decomposition above, 
with $\ell(Q') < \ell(Q)$, then 
$$ 
	\ell(Q') \geq \frac 1{3 \sqrt{2}} \dist(Q', \partial D)
\geq \frac 1{3\sqrt{2}} [ \dist(Q, \partial D) - \sqrt{2}  \ell(Q')]  
$$
which implies $\ell(Q') \geq \frac {1}{4\sqrt{2}} \ell(Q) 
	> \frac 18 \ell(Q)$. 
Since the side lengths are dyadic, we must 
have $\ell(Q') \geq \frac 14 \ell(Q)$. 
Thus adjacent squares differ in size by at most a factor of $4$.

\begin{lem} \label{strip triangles} 
Suppose $S=\{ x+iy :  0 < y < 2 \}$ is an infinite strip and the 
top and bottom edges are partitioned into segments 
of (Euclidean) length $\leq 1/8$ and that adjacent edges have 
lengths comparable to within a factor of $M$. 
Then there is a locally finite triangulation of the strip using only the
given boundary vertices and so that every angle 
of every triangle is $\geq \theta >0$ where $\theta$
only depends on $M$.   Thus the
triangulation has ``bounded degree'' depending 
only on $M$, in other words, the
number of triangles meeting at any vertex is 
uniformly bounded above by $2\pi/\theta$.
If both partitions are $L$-periodic (under horizontal 
translations) for some $L \geq 1$, 
then the triangulation is also $L$-periodic.
\end{lem} 

\begin{proof}
By splitting the strip into two parallel strips and rescaling, 
it suffices to consider the case when the top side is divided 
into unit segments (we triangular the top and bottom halves
separately and join them along a unit partition 
running down the center of the strip). 
The following argument is adapted from 
the proof of Theorem 3.4 in \cite{2021arXiv210316702B}.

If  $ \dots < x_{-1} < x_0 < x_1 < \dots$
are the partition points on the bottom edge define 
$$D_k = \min(|x_k-x_{k+1}|, |x_k-x_{k-1}|),$$ 
By assumption,
any two adjacent values of $D_k$ are comparable within 
a factor of $1\leq M < \infty$, and $\sup D_k \leq  1/8$.
Thus  $ 0< D_k/(16 M) \leq 1/128$ 
is contained in a dyadic interval of
the form $(2^{-j-1}, 2^{-j}]$
for some $j \geq 6$ (these half-open 
intervals form a disjoint cover of $(0, \infty)$).  
Let $y_k = \frac 34 \cdot 2^{-j}$ be the
center of this interval.  Note that $y_k$ and
$D_k/(16 M)$ are
comparable within a factor of $ \frac 32 < 2$,
so $y_k < D_k/(8 M) \leq \min(\frac 1{64}, D_k /8)$.

Let $z_k = x_k + i y_k$, $k \in \integers$
and consider the infinite polygonal arc  $\sigma$ with these vertices.
Note that $\sigma$ stays  within $1/64$ of the bottom edge
of the strip and every
segment has slope between $-1/8$ and $1/8$: the heights
of the endpoints above $x_k, x_{k+1}$  are each less than
	$$  \max(y_k, y_{k+1}) \leq
	\frac 18 \max(D_k , D_{k+1}) 
	\leq \frac 18 |x_k-x_{k+1}|,$$
so
$$ \frac {|y_{k+1}-y_k|} {|x_{k+1}-x_k|} \leq 
	 \frac {\max(y_{k+1},y_k)} {|x_{k+1}-x_k|} \leq  \frac 18.
$$

Tile the top half of $S$ by unit squares. Below this place 
a row of squares of side length $1/2$. Continue in this way, 
as illustrated in Figure \ref{Strip0}. We call this our 
decomposition of $S$ into dyadic squares. (This corresponds 
to the restriction of a Whitney decomposition of 
a half-plane to the strip.)

\begin{figure}[!htbp]
\centerline{
  \includegraphics[height=1.5in]{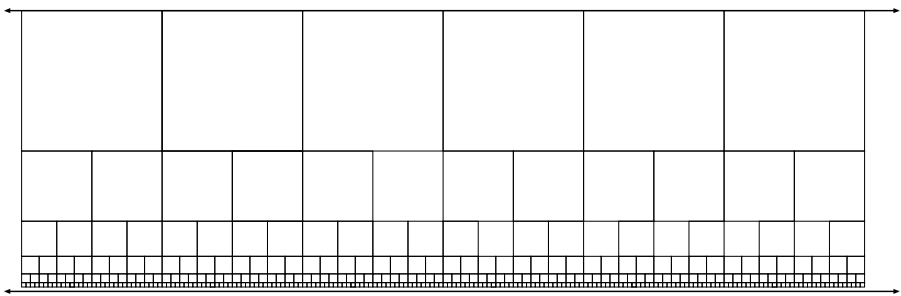}
}
\caption{ \label{Strip0}
The decomposition of $S$ into dyadic squares. 
}
\end{figure}

For each $k$, choose a square $Q_k$ from our
decomposition of the strip $S$ that contains $z_k$.
There is at least one decomposition  square containing 
$z_k$ since these squares  cover $S$, and there are 
at most two, since by 
our  choice of $y_k$, $z_k$ cannot lie on the top or bottom edge 
of any such $Q_k$ ($y_k$ was chosen to be halfway between 
these heights).  See Figure \ref{Qk}.
Let $I_k$ denote the vertical projection of $Q_k$ onto the 
bottom edge of $S$. 
Since the segments of $\sigma$ have  slope $\leq 1/8$, the height 
of $\sigma$ can change by at most $ \ell(Q_k)/8$ over  
$I_k$ and since it contains a point $z_k$ that is distance 
$\ell(Q_k)/2$ from both the top  and  bottom edges of 
$Q_k$,  $\sigma$ cannot intersect these edges of $Q_k$.  
Similarly, it cannot intersect the top or bottom edges of 
the adjacent dyadic squares of the same size as $Q_k$ 
that share the left and right edges of $Q_k$. In fact, it takes
at least horizontal distance $4 \ell(Q_k)$ for $\sigma$ to 
reach the height of the top or bottom of $Q_k$, so 
$\sigma$ does not intersect the top or bottom of the squares
that are up to three 
positions to the left or right on $Q_k$. This implies
that $\sigma$ does 
not intersect the ``parent'' square $Q_k^\uparrow$ of $Q_k$ 
(the square of twice
the size lying directly above $Q_k$), nor does it 
intersect the left or right neighbors of $Q_k^\uparrow$. 
 See Figure \ref{Qk}.
\begin{figure}[!htbp]
\centerline{
  \includegraphics[height=2.0in]{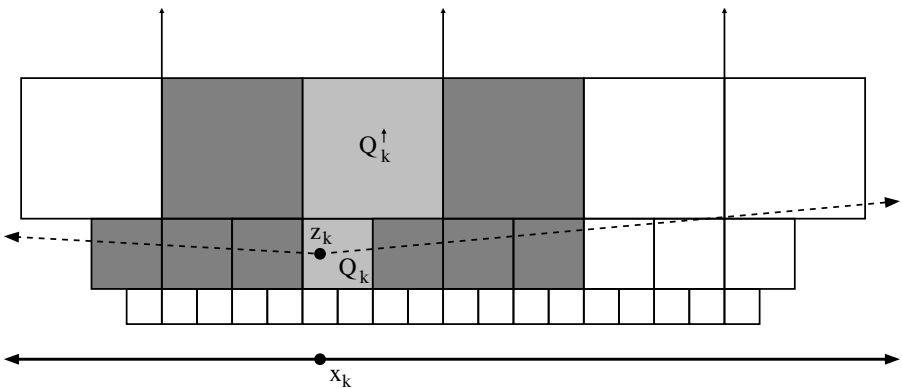}
}
\caption{ \label{Qk}
The point on the bottom edge is $x_k$, and 
above it is the corresponding $z_k$. The point 
$z_k$ is contained in a square $Q_k$ and above 
this is its ``parent'' $Q_k^\uparrow$ (both 
lightly shaded).
The dashed  curve is part of $\sigma$. Note that $\sigma$ 
intersects at least three squares to the 
left and right of $Q_k$ (darker shading). 
This implies the  ``parent''  square
$Q_k^\uparrow$ does not intersect $\sigma$, nor do the squares 
	to the left or right of the parent (also dark shaded). 
}
\end{figure}

Now remove all the squares whose interiors  intersect $\sigma$ 
or that lie below $\sigma$.
The set of  remaining squares  
contains the whole top row of unit squares.   Since 
$\sigma$ has small slope, if a square $Q$ is above $\sigma$, so is its parent (and by induction,  all its ancestors).
Let $\gamma$ denote the lower boundary of union 
See the top of  Figure \ref{Strip3}.
of remaining squares; this is a locally polygonal 
curve made up of horizontal and vertical segments. 
A vertex of $\gamma$ is any corner of a decomposition 
square that lies on $\gamma$, and a corner of $\gamma$
is a vertex where a horizontal and vertical edge of 
$\gamma$ meet.
Let $W$ denote the infinite region bounded above 
by $\gamma$ and below by the bottom edge of $S$
(shaded region in top picture of Figure \ref{Strip3}).

\begin{figure}[!htbp]

\centerline{
  \includegraphics[height=1.35in]{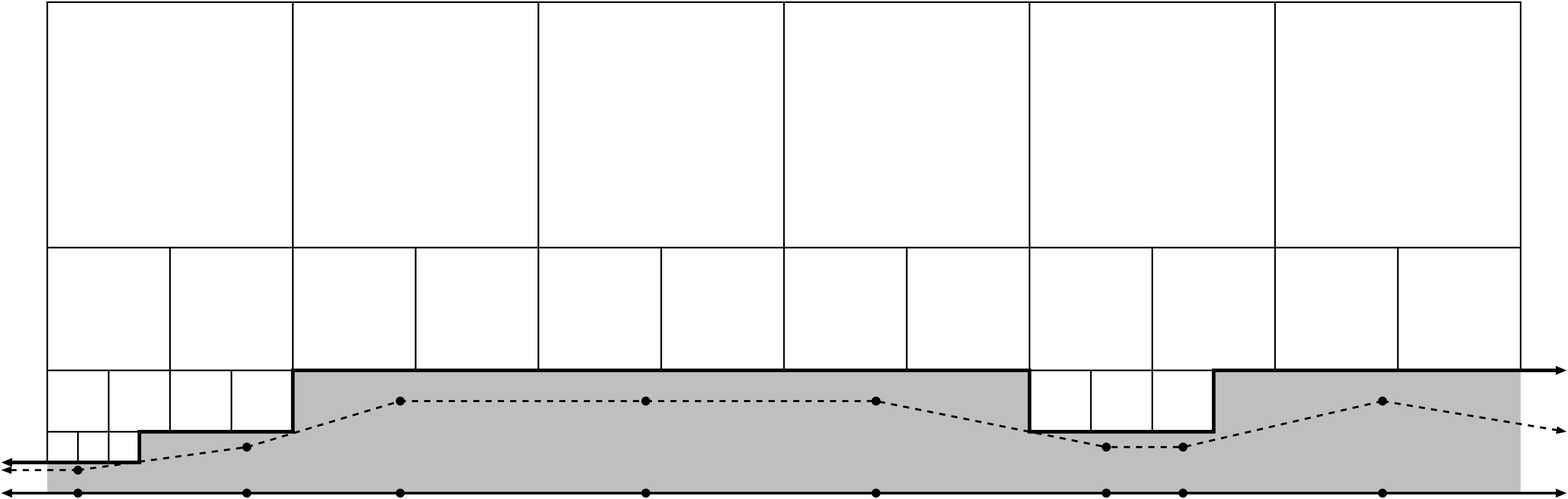}
}
\vskip.2in
\centerline{
  \includegraphics[height=1.35in]{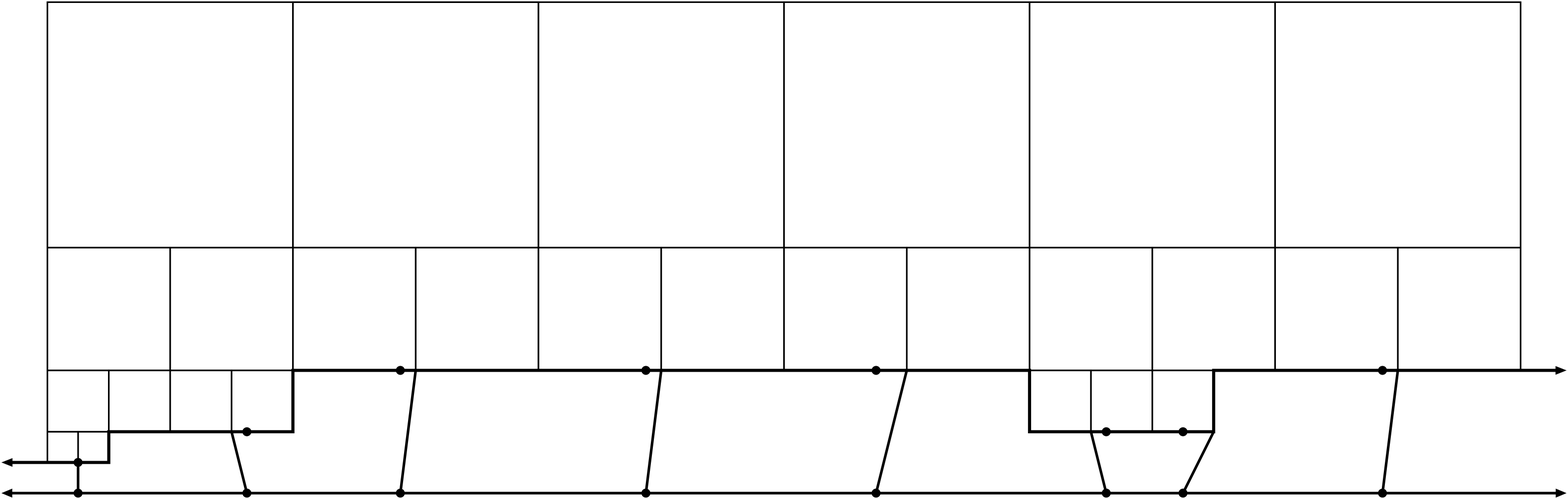}
}
\vskip.2in
\centerline{
  \includegraphics[height=1.35in]{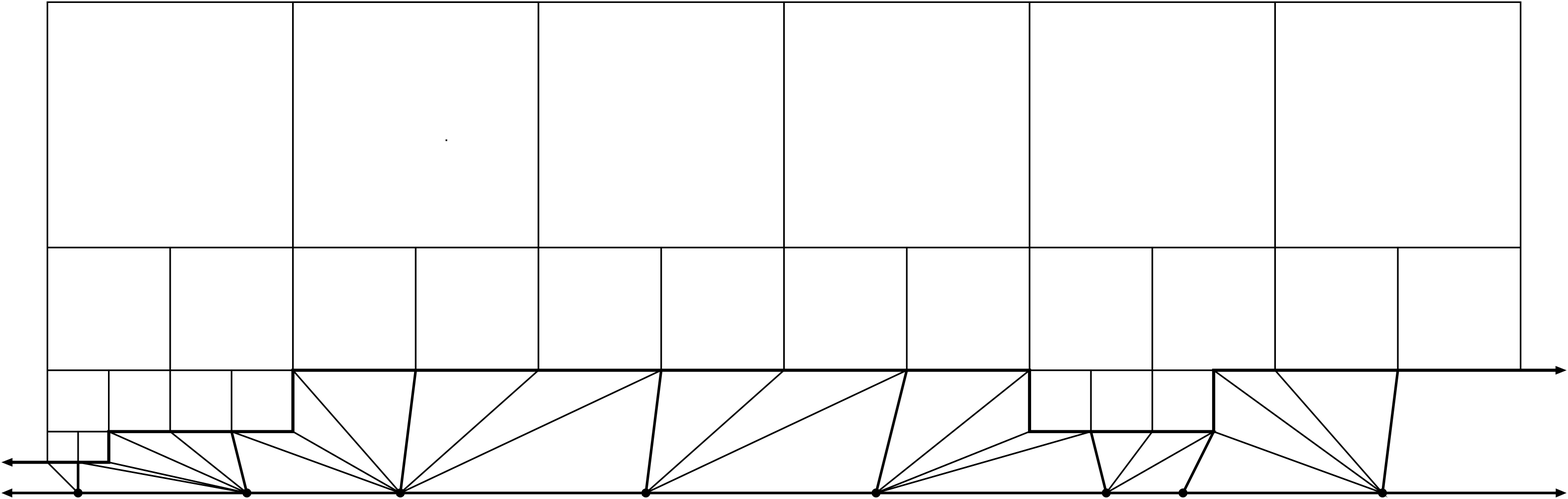}
}
\vskip.2in
\centerline{
  \includegraphics[height=1.35in]{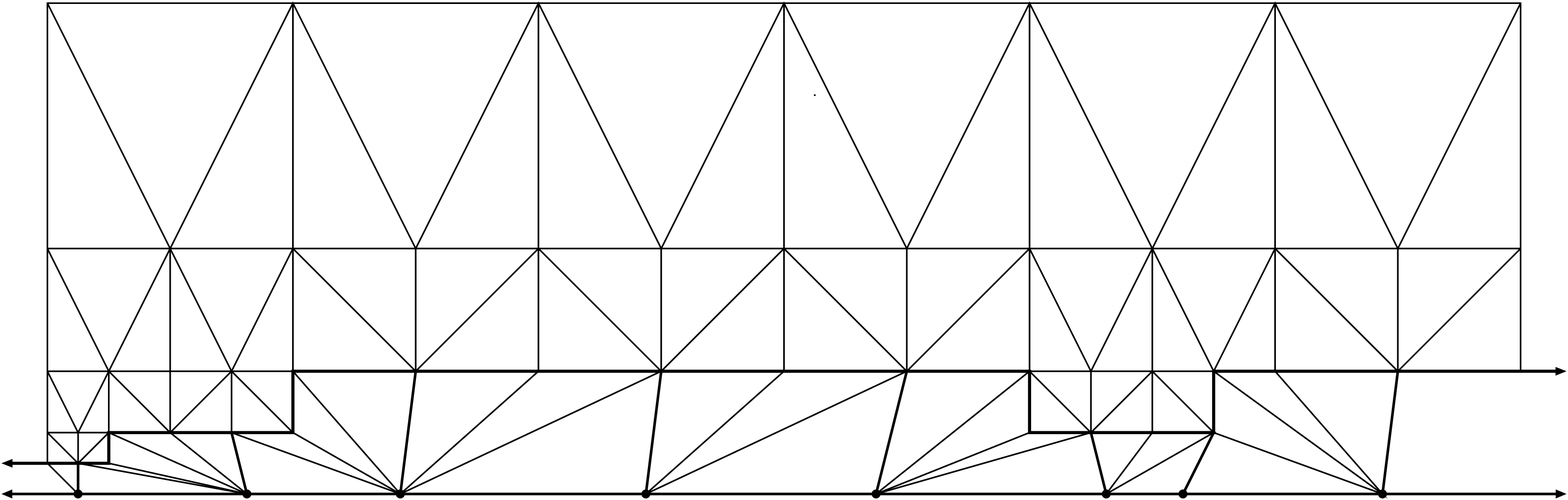}
}
\caption{ \label{Strip3}
The top figure shows the region  $W$ (shaded) below 
$\gamma$. The second figure  divides $W$ into 
quadrilaterals by connecting each $x_k$ to a vertex
of $\gamma$ that is closest to being ``above'' $x_k$. 
We then  triangulate the quadrilaterals by connecting all vertices 
of $\gamma$ to either the lower left or lower right corner, 
depending on whether $\gamma$ is  decreasing or increasing  
between $x_k$ and $x_{k+1}$. 
The bottom picture shows  the 
squares above $\gamma$ triangulated in the obvious way. 
}
\end{figure}

Let $\gamma_k$ be the subarc
of $\gamma$ that projects onto $[x_k, x_{k+1}]$.
By construction, each $x_k$ lies below the parent
of $Q_k$, and the squares to the left and right of 
the parent are also above $\sigma$, so $x_k$ is 
at least distance $2 \ell(Q_k)$ from the vertical 
projection of any corner of $\gamma$. 
Connect $x_k$ to a vertex  $w_k$ of $\gamma$ whose vertical 
projection is closest to $x_k$, or  to either one in case 
of a tie. Note that $w_k$ is a vertex  on the bottom 
edge of $Q_k^\uparrow$; a tie occurs only if $w_k$ is the 
midpoint of this bottom edge. Adding the segments from 
$x_k$ to $w_k$ divides $W$ into quadrilaterals. See 
the second figure in Figure \ref{Strip3}.

Over the interval $(x_k, x_{k+1})$, the polygonal curve 
$\gamma$ is either a horizontal segment, 
a decreasing stair-step or an increasing 
stair-step. In the first two cases, connect  every vertex 
of $\gamma$ between $w_k$ and $w_{k+1}$ (including these 
points) to $x_k$. In the third case, connect them all 
to $x_{k+1}$. In either case, this triangulates $W$ 
with triangles so that all three edges have comparable 
lengths  and no angle is close to $180^\circ$, so by 
the Law of Sines, all the angles are bounded uniformly 
away from $0$ (the bound depends on $M$,  the constant of 
comparability between adjacent arcs on the boundary of $S$).
\end{proof}

The following simple lemma will allow us to 
build equilateral triangulations from topological 
triangulations that are ``close to'' equilateral 
in a precise sense.

\begin{lem} \label{equilateral2}
Suppose $K < \infty$ and  ${\mathcal T}$ is a topological 
triangulation of a domain $\Omega$ and for each
triangle $T \in {\mathcal T}$, there is a 
$K$-quasiconformal map $f_T$  sending $T$ to a Euclidean 
equilateral triangle and that is length multiplying on 
each boundary edge. Let $\mu_T$ be the dilatation of $f_T$.
If $f$ is a quasiconformal map 
on $\Omega$ with dilatation  $\mu_T$ on $T$, 
then $f({\mathcal T})$ is an equilateral triangulation 
of $f(\Omega)$. 
\end{lem} 

\begin{proof}
We use the characterization of equilateral triangulations 
given in Lemma 2.5 of \cite{2021arXiv210316702B}: a triangulation 
of a Riemann surface is equilateral iff given any two 
triangles $T, T'$ that share an edge $e$, there is an 
anti-holomorphic homeomorphism $T\to T'$ that fixes 
$e$ pointwise, and maps the vertex $v$ opposite $e$ in $T$ 
to the vertex $v'$ opposite $e$ in $T'$. 

For any two  triangles  $T_1, T_2$
in $f({\mathcal T})$ that are adjacent along an edge $e$, 
define  $g = \iota_k \circ f_{T_{k}} \circ f^{-1}$ 
on $T_k$, $k=1,2$, where  $\iota$ is an appropriately 
chosen similarity of the plane to make the image 
triangles match up along the segment $I$ that is the image 
of $e$. By the length multiplying property of the 
maps $f_T$, $g$ is continuous across $e$. 
Then $g^{-1} \circ R \circ g$, where $R$ is 
reflection across $I$, is the anti-holomorphic maps that 
swaps $T_1$ and $T_2$ as required. 
\end{proof} 

The image triangulation ${\mathcal T}'$ will be close 
to ${\mathcal T}$ if the dilatation $\mu$
is close to zero in an appropriate sense. 
For our applications below, this will mean that the 
dilatation of $|\mu|$ is uniformly bounded below $1$ 
and that the support of $\mu$ has small area. As the
area tends to zero, $f$ can be taken to uniformly approximate 
the identity, and so ${\mathcal T}'$ approximates 
${\mathcal T}$ as closely as we wish.

The following is elementary and left to the reader.  See 
Figure \ref{Affine} for a hint. 

\begin{lem}  \label{affine lemma} 
Any Euclidean  triangle $T$  can be uniquely  mapped to 
a equilateral triangle $T'$ by an affine map by specifying 
a distinct vertex of $T'$ for each vertex of $T$.
This map is $K$-quasiconformal where $K$ 
depends only on  the minimal angle of $T$. 
\end{lem}

\begin{figure}[!htbp]
\centerline{
  \includegraphics[height=1.0in]{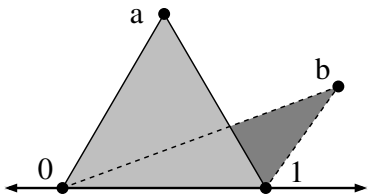}
}
\caption{ \label{Affine}
To compute the dilatation of affine maps between triangles, place 
both triangles with one edge $[0,1]$ that is fixed 
by the map, and opposite vertices $a,b$.   
The affine map has the form $z\to \alpha z + \beta \overline{z}$. 
Since $0,1$ are fixed, we can solve for 
$\alpha,\beta$  and this gives 
$|\mu| = |\beta/\alpha| = |(b -a)/(b-\overline{a})|$.
This is bounded below $1$ iff the angles of the 
triangle with vertices $0,1,b$ are bounded away from zero.
}
\end{figure}


\begin{lem} \label{annuli triangulation}
There is a constant $C< \infty$ so that the 
following holds. 
Suppose $A$ is a conformal grid annulus, 
and $f: A \rightarrow B=\{ 1< |z| < 1+\delta\}$ is a conformal mapping, where 
$\delta \leq 1/100$.
Suppose also that $\textrm{length}(f(I))<\delta/10$ for each sub-arc $I$ of $A$.
Then $A$ has a topological triangulation 
such that each triangle $T$ in the 
triangulation can be  
mapped to a  equilateral triangle  by a 
$C$-quasiconformal map that multiplies arclength 
on each side of $T$, and the degree of any vertex is bounded by a universal constant (independent of $A$).
\end{lem} 

\begin{proof}
Use the logarithm map (and a rescaling) to lift the partition of $\partial B$ to a 
partition of $\partial S$ where 
$S = \{ x+iy: 0 < y < 2\}$.
The resulting segments all have length $\leq 1/8$, so Lemma \ref{strip triangles}
applies to give a triangulation of $S$. Moreover, the degree of any vertex in this triangulation is bounded by a universal constant by Lemma \ref{strip triangles}, since for any adjacent sub-arcs $I$, $J$ on $A$, by Lemma \ref{comparable arcs} we have that the lengths of $f(I)$, $f(J)$ are comparable with a uniform constant (independent of $A$).


By Lemma \ref{affine lemma}, each triangle in our triangulation  of the 
strip can be 
uniformly quasiconformally mapped to an equilateral triangle 
by a map that multiplies arclength on each edge. Thus for two 
triangles sharing an edge, and mapping to equilateral triangles 
that share the corresponding edges, the maps agree along the common 
edge. Pulling this   periodic dilatation back  to $B$ 
via exponential map preserves the size of the dilatation (since
the  map is conformal).
We then pull the triangulation back to 
$A$ via the quasiconformal 
map $ \psi:A \to B$  given by Lemma \ref{length preserving}. 
This gives a  smooth triangulation of $A$ 
and a dilatation  $\mu$  on $A$   that is uniformly bounded 
(since the dilatation of $\psi$ is) and that 
transforms the triangulation into an equilateral 
triangulation under any quasiconformal map of $A$ that 
has dilatation  $\mu$ on $A$ by Lemma \ref{equilateral2}.
\end{proof} 

We will also want to bound the sizes of the triangles
produced in the previous lemma. We will do this using 
estimates of harmonic measure and the hyperbolic metric, the definitions of which we now briefly recall. The hyperbolic metric $\rho$ on $\mathbb{D}:=D(0,1)$ is defined infinitesimally by 
\begin{equation} \rho(z)|dz|:=\frac{|dz|}{1-|z|^2}. \end{equation} 
Any domain $\Omega$ satisfying $|\Chat\setminus \Omega|>2$ is \emph{hyperbolic}, in other words the universal cover of $\Omega$ is $\mathbb{D}$, and the covering map $\phi: \mathbb{D} \rightarrow \Omega$ defines the hyperbolic metric $\rho$ on $\Omega$ via the equation: 
\begin{equation} \rho(w)|dw|:=\frac{|\phi'(z)||dz|}{1-|z|^2},\hspace{2mm} \phi(z)=w,\end{equation}
(see for instance Exercise IX.3 in \cite{MR2150803}).

We will consider harmonic measure only in simply connected domains with locally connected boundary, where the definition is as follows (see also the monograph \cite{MR2150803}). First, for an interval $I\subset\mathbb{T}$, we simply define 
\begin{equation} \omega(0, I, \mathbb{D}):=\textrm{length}(I)/2\pi.\end{equation}
 If $\Omega$ is a simply connected domain with locally connected boundary, we define harmonic measure in $\Omega$ by pulling back under a conformal map $\phi: \mathbb{D} \rightarrow \Omega$. More precisely, if $\phi: \mathbb{D} \rightarrow \Omega$ is a Riemann mapping, and $I\subset\mathbb{T}$ is an interval, then we define the harmonic measure of $J:=\phi(I)$ with respect to $w:=\phi(0)$ in $\Omega$ by the formula:
 \begin{equation} \omega(w, J, \Omega):=\omega(0, I, \mathbb{D})= \textrm{length}(I)/2\pi.\end{equation}

\begin{rem}\label{hyperbolic_metric_remark}
If $\Omega$ is a simply connected domain then the 
hyperbolic metric $\rho$ in $\Omega$ satisfies 
the well known estimate 
$$  \frac {1}{4 \cdot \dist(z, \partial \Omega)} \leq 
\rho(z) \leq \frac {1}{\dist(z, \partial \Omega)}.$$
See, e.g., equation (I.4.15) of \cite{MR2150803}.
More generally, we have 
$$   
\rho(z)  \simeq \frac {1}{\dist(z, \partial \Omega)}$$
for multiply connected domains with uniformly perfect
boundaries. A set $X$ is uniformly perfect  if
there is a constant $M < \infty$ so that for every 
$0 < r< \diameter (X)$ and every $x\in X$ there is a
$y \in X$ with $ r/M \leq |x-y| \leq r$. All round annuli $B = \{ 1< |z| < 1+\delta\}$ considered here have this property with uniform $M$.


\end{rem}

\begin{lem} \label{harmonic} 
Suppose $S =\{x+iy:  0 < y < 1\}$ and  $I$ is an arc on the 
bottom edge of $S$ with $\ell(I) \leq 1/2$. 
Suppose $\varepsilon>0$ and  $z= x+iy \in S$  with 
	$ \varepsilon \cdot  \dist( x, I) \leq 
	y \leq  \min( \frac 12, \ell(I)/\varepsilon) $ .
Then the harmonic measure of $I$ in $S$  with respect to $z$
satisfies $ \omega(z, I, S) \geq \delta(\varepsilon) >0$.
\end{lem} 

\begin{proof} 
Let $T$ be the right isosceles triangle with 
hypotenuse $I$. See Figure \ref{HarmonicMsr1}.
Then the harmonic measure of $I$ 
in $S$ with respect to a point in $T$ is greater than 
its harmonic measure in the square $Q$ with base $I$, and the 
latter is easily checked to be $\geq 1/4$ in $T$. 
Moreover, our conditions imply $z$ is   a bounded  
hyperbolic distance (in $S$)  from $T$, with 
a bound depending only on $\varepsilon$. Thus  by Harnack's 
inequality, the harmonic measure of $I$ with respect  $z$ 
is comparable to $1/4$, e.g., is bounded uniformly 
away from zero in terms of $\varepsilon$. 
\end{proof}

\begin{figure}[!htbp]
\centerline{
  \includegraphics[height=1.5in]{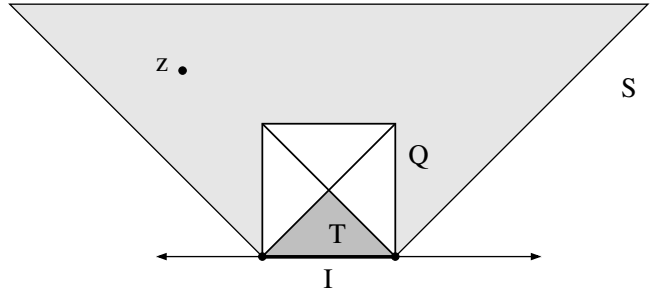}
}
\caption{ \label{HarmonicMsr1}
The harmonic measure of $I$ in the square with base 
$I$ is at least $1/4$ in all points of the shaded triangle. 
Hence it is at least $1/4$ in the strip containing the 
square. Thus it is $\simeq 1$ at any point within bounded 
hyperbolic distance of the shaded triangle. 
}
\end{figure}

\begin{cor} \label{triangle size}
The triangulation ${\mathcal T}$ of $A$  given by Lemma 
\ref{annuli triangulation} has the following 
properties. If $T \in {\mathcal T}$    does
not touch $\partial A$, then 
$$ \diameter(T)  \leq C' \max \{ \dist(z,\partial A): z \in A\}
	= O({\rm{inrad}}(A)),$$
for some fixed $C' < \infty$. 
If $T \in {\mathcal T}$  has one side $I$ on $\partial A$, then
	$$ \diameter(T)  \leq C' \diameter(I) =O({\rm{gap}}(A)).$$
This estimate also holds if 
$T \in {\mathcal T}$  has only one vertex on $\partial A$
and this vertex is the endpoint of a sub-arc $I \subset 
\partial A$. 
\end{cor} 

\begin{proof}
By the explicit construction given in 
the proof of Lemma \ref{strip triangles}, 
any interior triangle is contained in a Whitney 
square for the strip, and so has uniformly bounded 
hyperbolic diameter in the strip. Quasiconformal 
maps are quasi-isometries of the hyperbolic 
metric; for a sharp version of this, see Theorem 5.1 of 
\cite{MR2052356}. 
Therefore the hyperbolic diameter 
of the image triangle $T$ in $A$
is also uniformly bounded. Hence the standard
estimate of hyperbolic metric discussed above (see Remark \ref{hyperbolic_metric_remark}) shows that 
$$
\diameter(f(T)) \leq C' \dist(T, \partial A) =O( {\rm{inrad}}(A)) .$$ 

On the other hand,  our construction implies that 
if $T \subset S$ is associated to a sub-arc $I \subset
\partial S$, in either of the two ways described in 
the current lemma, then by 
Lemma \ref{harmonic} we have $\omega(z,I,S) \geq \varepsilon >0$,
in other words, the harmonic measure of $I$ 
with respect to any point $z\in T$ is uniformly bounded 
above zero by a constant $\varepsilon$ that only depends 
on the comparability constant $M$ in 
the proof of Lemma \ref{strip triangles}. 
If we conformally  map the strip $S$ to the unit 
disk with $z$ going to the origin, this means that 
$I$ maps to an arc $J$  on the unit circle whose length 
is bounded uniformly away from zero. 

Now consider the path family of arcs in $\disk$ 
with both endpoints on $J$ that separate 
$0$ from $\circle \setminus J$. This has 
modulus that is bounded away from zero, since 
the length of $J$ is bounded below.  By 
the conformal invariance of modulus, 
the corresponding family in the strip $S$ has 
modulus bounded away from zero, and by 
quasi-invariance so does the image of this 
family in $A$. Now suppose by way of contradiction that $\dist(f(z), f(I))  \not= O(\diameter(f(I)))$. Then the modulus of this family would be small: this can be seen by comparing it to the modulus of the paths
connecting the two boundary components of a round 
annulus with inner boundary a circle of radius $\diameter(f(I))$ and outer boundary a circle of radius $\dist(f(z), f(I))$. This is a contradiction, and thus we conclude that $\dist(z,f(I)) \leq M\cdot \diameter (f(I))$ for 
some fixed $M< \infty$, as desired. 

\end{proof}

\section{Triangulating Domains} 
\label{Triangulating_domains}

In Section \ref{Triangulating_domains} we prove Theorem \ref{theorem_B} following the inductive approach described in the Introduction. We start our construction of an  equilateral  triangulation of 
a planar domain $D$ 
with the following  lemma for surrounding a 
compact set with well separated contours. 

\begin{lem}  \label{surround}
Given a compact set $K \subset \complex$, there
are sets $\Gamma_n$ so that for all $n \in \naturals =
	\{ 1,2,3, \dots \} $ we have 
\begin{enumerate}
\item  each $\Gamma_n$ is  made up of a 
finite number of axis-parallel, simple polygons,
\item  each $\Gamma_n$ separates $K$ from $\infty$ and separates 
$\Gamma_{n+1}$ from $\infty$, 
\item   $    16^{-n} \leq \dist(z, K) \leq 3 \cdot 16^{-n}$
	 for every  $z \in \Gamma_n$, 
\item   $d_n=\dist(\Gamma_n, \Gamma_{n+1} )  \geq 13\cdot 16^{-n-1}$,
\item different connected components of $\Gamma_n$ are at 
	least (Euclidean) distance $ 2 \cdot 16^{-n-1}$ apart.
\end{enumerate} 
\end{lem} 

\begin{proof} 
Let $D$ be the unbounded connected component 
of $\complex \setminus K$. This is an 
unbounded domain with compact boundary contained in $K$. Let ${\mathcal W}$ be the family of dyadic squares defined at the beginning of Section \ref{Triangulating_strips}.
For $n=1,2,3,\dots$, let $D_n$ be the union of all
(closed) squares in ${\mathcal W}$  that intersect 
$\{z \in D: \dist(z,\partial D) \leq 16^{-n}\}$.
Each chosen square has distance $\leq 16^{-n}$
from $\partial D$, so by 
(\ref{bounds}), all the chosen squares have side 
lengths between $16^{-n-2}$ and  $ 16^{-n}$.
Let $\Gamma_n = \partial D_n \cap D = \partial D_n \setminus 
\partial D$. Then $\Gamma_n$ is a union of axis-parallel
polygonal curves and  each segment in $\Gamma$ is on 
the boundary of a square not in $D_n$ and therefore 
$$  16^{-n} \leq \dist(z, \partial D)$$ 
for every $z \in \Gamma_n$. 
See Figure \ref{Whitney}.

\begin{figure}[!htbp]
\centerline{
  \includegraphics[height=3in]{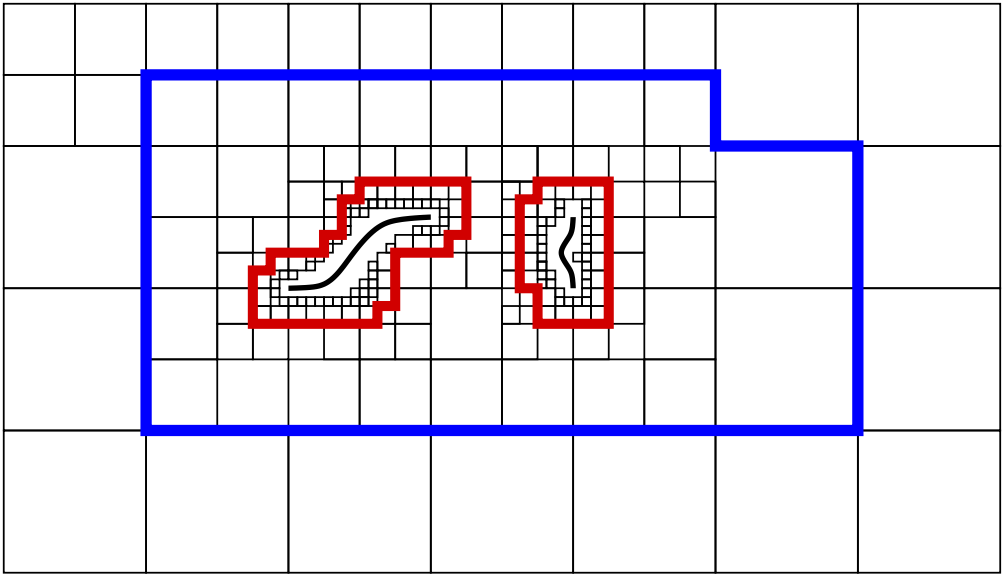}
}
\caption{ \label{Whitney}
An example of a Whitney decomposition of the complement 
of a compact set $K$. By using boundaries of unions 
of Whitney boxes, we can create polygonal contours 
that surround $K$ at approximately constant distance.
}
\end{figure}

On the other hand, every segment in $\Gamma$ is 
on the boundary of a square $Q$  inside 
$D_n$,  and hence  for every $z \in \Gamma_n$ we have
$$ \dist(z,  \partial D) 
\leq 16^{-n} + \diameter(Q) 
\leq 16^{-n} +\sqrt{2} \cdot 16^{-n} < 3 \cdot 16^{-n}.$$ 
Thus (3) holds.  To prove (4), note that 
$$  \dist(\Gamma_n,  \Gamma_{n+1}) \geq 16^{-n} - 
3 \cdot 16^{-n-1} = 13 \cdot 16^{-n-1}.$$ 

It remains to prove (5). If a connected component of $\Gamma_n$ is not a simply polygon, 
it is because there is a point $x \in \Gamma_n$ so that  exactly two 
squares $Q_1$, $Q_2$ 
intersecting $\{\dist(z, \partial D) = 16^{-n}\}$  both contain $x$ 
as corners, but these two squares do not share edge, in other words, 
$\Gamma_n$ looks like a cross at $x$. We can replace the cross by 
two disjoint arcs passing through the centers of $Q_1, Q_2$, 
as shown in Figure \ref{Simple}. Doing this (at most 
finitely often) makes each connected component of $\Gamma_n$ 
a simple  polygon, every segment of which 
has length $\geq 2\cdot 16^{-n-1}$.

\begin{figure}[!htbp]
\centerline{
  \includegraphics[height=2.0in]{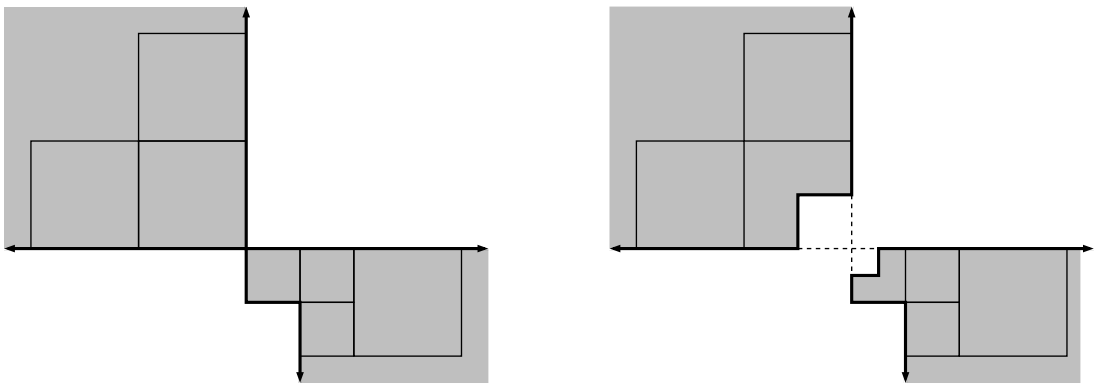}
}
\caption{ \label{Simple}
We can assume components of $\Gamma_n$ are simple curves 
by removing any self-intersections  at a point $x$ as shown.
The distance between the new curves is at least half 
the side length of the smaller square $Q$ intersecting $x$; 
by our estimates $\ell(Q) \geq \frac 14 16^{-n}$.
}
\end{figure}

Finally, any decomposition square that
is adjacent to $\Gamma_n$  contains a point 
at distance $\geq 16^{-n}$, for otherwise it would 
be contained in the interior of $D_n$ and every 
surrounding square would intersect $D_n$. Hence such a
square has side length $\geq \frac 14 \cdot  16^{-n}$.
Since any two distinct  components of $\Gamma_n$
are separated by a 
collection of such squares, the two components are 
separated by at least $ \frac 14 \cdot 16^{-n}$. 
If the modification in the last paragraph creates 
two separate components, then these components are 
at least $\frac 18 \cdot 16^{-n} = 2 \cdot 16^{-n-1}$ apart.
\end{proof}

We will build the desired 
triangulation using an inductive construction. 
The first step is given by the following lemma.

\begin{lem} \label{first step}  
For any $\varepsilon >0$ there is a (finite) 
equilateral triangulation ${\mathcal T}_0$
of the Riemann sphere so that 
\begin{enumerate} 
\item every triangle has spherical diameter $< \varepsilon$, 
\item the part of the triangulation contained in 
the unit disk is the conformal image 
of a Euclidean equilateral triangulation of 
some equilateral grid polygon under a conformal 
map  $f$ with $\frac 12 \leq |f'| \leq 2$.
\end{enumerate}
\end{lem}

\begin{proof}
The four sides of a equilateral tetrahedron give 
an equilateral triangulation of the sphere. By 
repeated dividing each Euclidean triangle into 
four smaller equilateral triangles, we may make
every triangle on the sphere as small as we wish. 
If we normalize so that one side of the original 
tetrahedron covers  a large disk around the 
origin, then the second 
condition above is also satisfied.
See Figures \ref{equi_tetra} and \ref{equi_sphere}.
\end{proof} 

\begin{figure}[!htbp]
	\begin{minipage}{1.5in}
\centerline{
  \includegraphics[trim = 0 0 0 0, clip, height=1.25in]{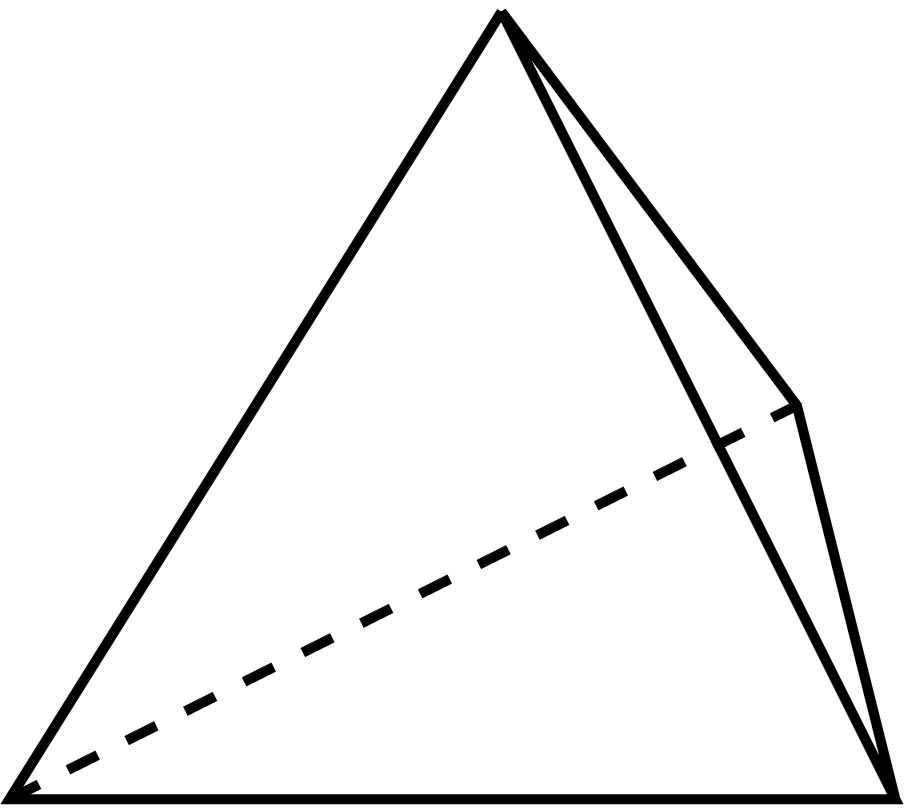}
}
\centerline{
  \includegraphics[trim = 0 0 0 0, clip, height=1.5in]{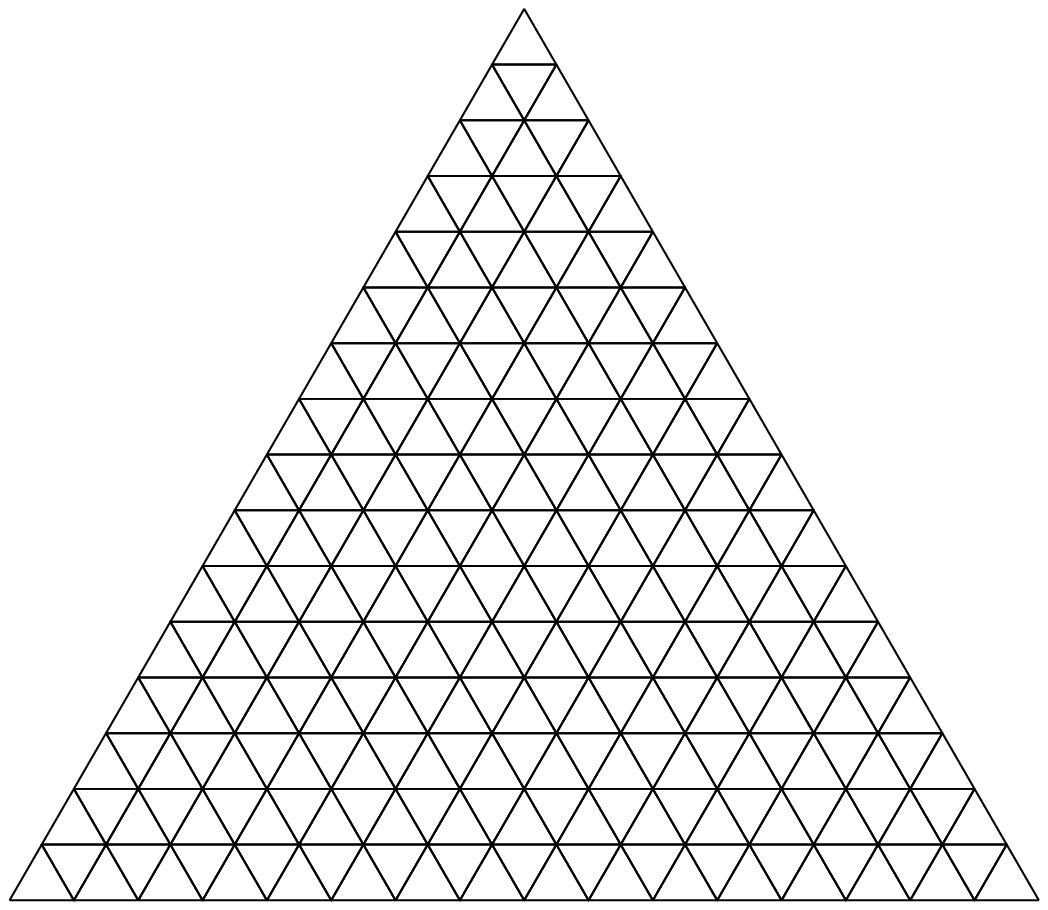}
}
\end{minipage}
	\hphantom{xx}
\begin{minipage}{3in}
\centerline{
  \includegraphics[trim = 350 75 300 50, clip, height=3in]{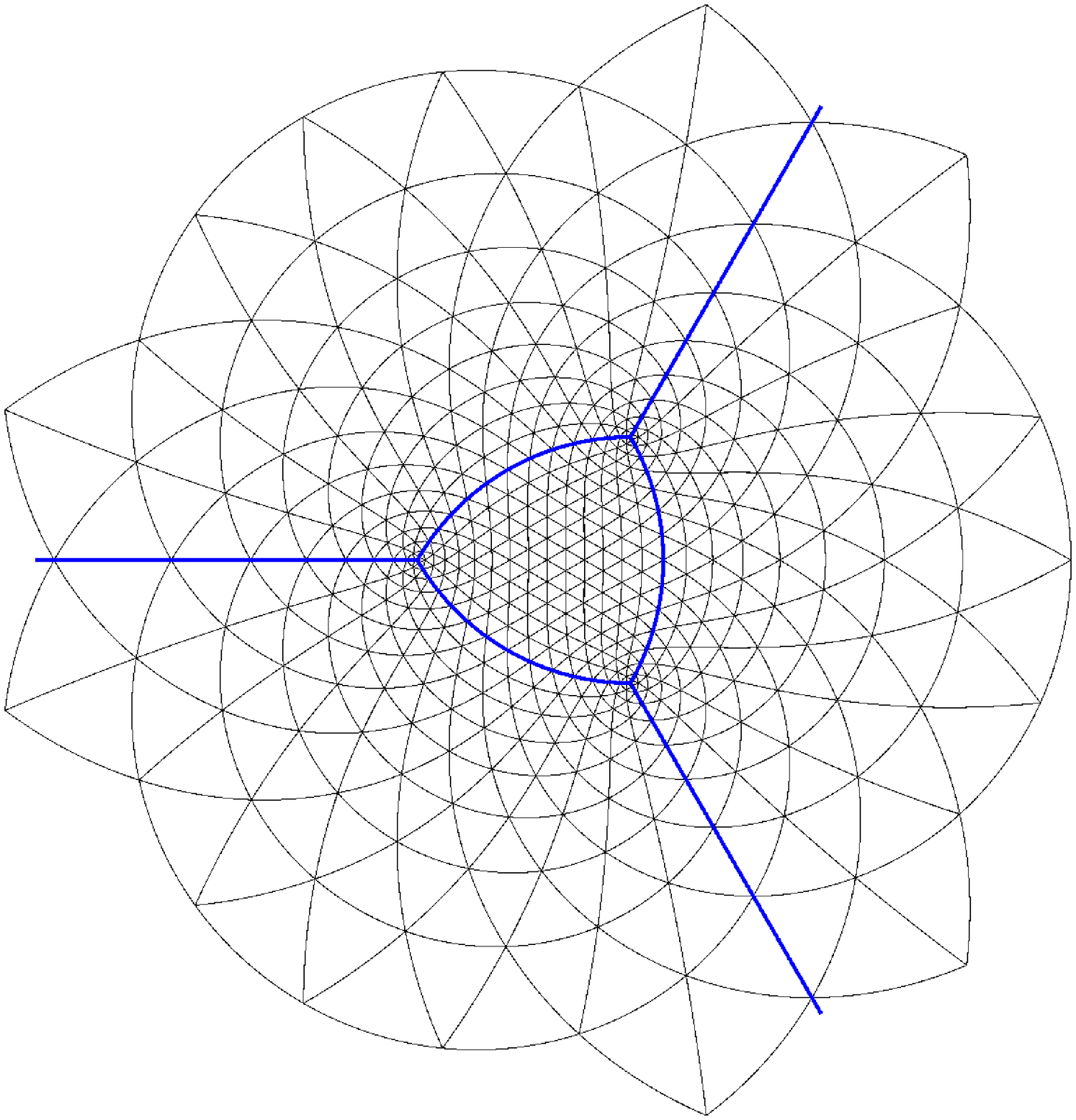}
}
\end{minipage}
\caption{ \label{equi_tetra}
An equilateral tetrahedron with the flat metric 
on each side can be conformally mapped 
to the sphere by the uniformization theorem. 
Here we plot part of the image
in the plane: the thick edges are the images 
of the edges of the tetrahedron, and the triangulation 
is invariant under reflection in these edges. The 
center region is a Reuleaux triangle with 
interior angles of $120^\circ$ (each edge 
is a circular arc centered at the opposite vertex). 
See Figure \ref{equi_sphere} for the 
same triangulation drawn on a sphere.
}
\end{figure}

\begin{figure}[!htbp]
\begin{minipage}{3in}
\centerline{ \includegraphics[trim = 270 150 270 130, clip, height=3in]{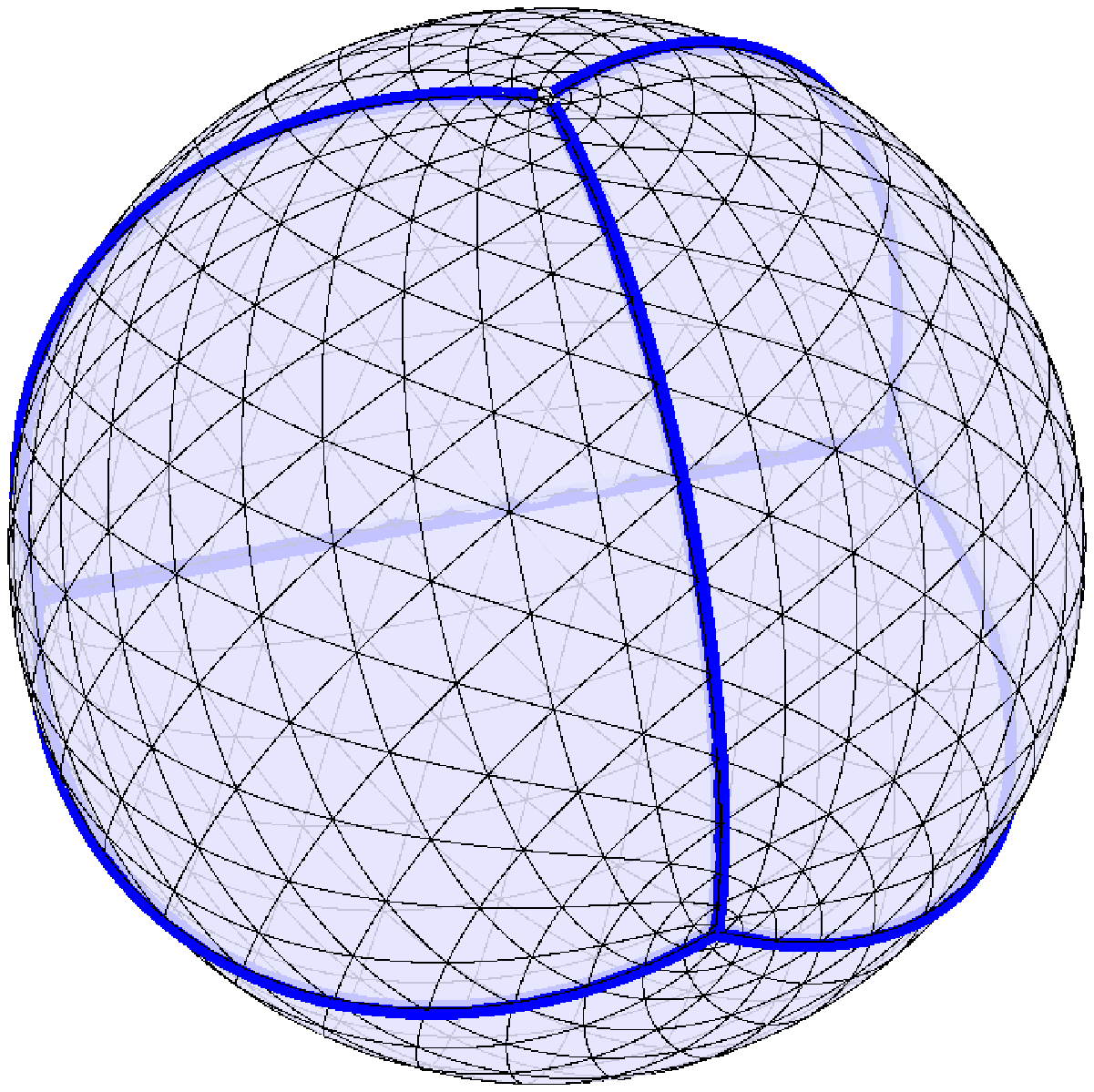}
}
\end{minipage}
\caption{ \label{equi_sphere}
The equilateral triangulation from Figure 
\ref{equi_tetra} projected stereographically 
onto the sphere.
}
\end{figure}

\noindent \emph{Proof of Theorem B. } Let $D$, $\eta$ be as in the statement of Theorem B. We claim that it suffices to prove the Theorem in the special case that 
\begin{equation}\label{WLOG_assumption} \infty\in D\textrm{ and } K:=\partial D\subset D(0,1/16). \end{equation}
Indeed, if we are then given an $\eta$ and a domain $D$ which does not satisfy (\ref{WLOG_assumption}), we may apply a M{\"o}bius transformation $M$ (defined by a spherical isometry moving a point in $D$ to $\infty$, followed by a scaling map $z\mapsto\lambda z$) so that $M(D)$ satisfies (\ref{WLOG_assumption}). Applying the special case of the Theorem to $M(D)$ and an appropriately rescaled version of $\eta$ then gives a triangulation $\mathcal{T}$ of $M(D)$ so that $M^{-1}(\mathcal{T})$ is the desired triangulation of $D$. Henceforth, we assume (\ref{WLOG_assumption}).



Let $\{\Gamma_n\}_{n=0}^\infty$ be the polygonal contours surrounding $\partial D$ obtained by applying Lemma \ref{surround} to $K=\partial D$. Fix an $N\geq20$ so that $N/2$ satisfies the conclusions of Lemma \ref{thickness} with $\varepsilon=1/10$. Let 
\begin{equation}\label{Unepsilondefn} U_{n, \varepsilon} :=\{z: \dist(z, \Gamma_n) \leq  N \cdot \varepsilon\}. \end{equation}


We will now fix a sequence $(\varepsilon_n)_{n=1}^\infty$ by specifying each $\varepsilon_n$ be sufficiently small so as to satisfy the following finite set of conditions. First, let $C>0$ be the maximum of the constant $C$ in Lemma \ref{annuli triangulation} and $K$ in Lemma \ref{boundary triangles}. As argued in the proof of Proposition \ref{phi_close_to_id} (see also Lemma 2.1 of \cite{MR3232011}), there exists a constant $a_n>0$ such that any $C$-quasiconformal mapping $\phi:\mathbb{C}\rightarrow\mathbb{C}$ normalized to fix $0$, $1$ whose dilatation is supported on a region of (Euclidean) area $<a_n$ satisfies
\begin{equation}\label{phismallpert} d(\phi(z),z)<16^{-n-2} \textrm{ for all } z\in\mathbb{D}. \end{equation} 
We specify $\varepsilon_n$ be small enough so that $U_{n, \varepsilon_n}$ has area $<a_n$, and set
\begin{equation}\label{Undefn} U_n:=U_{n, \varepsilon_n}=\{z: \dist(z, \Gamma_n) \leq  N \cdot \varepsilon_n\}. \end{equation}

Next, we note that any $C$-quasiconformal map $\phi:\mathbb{C}\rightarrow\mathbb{C}$ normalized as above is H\"older continuous with uniform bounds (see for instance Section I.4.2 of \cite{MR0344463}). In particular there exist constants $M$, $\alpha>0$ so that
\begin{equation}\label{holderestimate} \diam(\phi(E)) \leq M\diam(E)^\alpha \textrm{ for any } E\subset\Chat\end{equation}
for any normalized $C$-quasiconformal map $\phi:\mathbb{C}\rightarrow\mathbb{C}$.
Let $C'<\infty$ be as in the conclusion of Corollary \ref{triangle size}. By (\ref{holderestimate}) and the Lipschitz-equivalence of the spherical and Euclidean metrics on $\mathbb{D}$, we may specify that $\varepsilon_n$ be small enough so that if $E\subset\mathbb{D}$, then:
\begin{equation}\label{forthediamestimate} \diameter(E)<NC'\varepsilon_n \implies \diam(\phi(E)) < \eta(16^{-n-2}).  \end{equation}
Lastly, we specify that $\varepsilon_n$ be sufficiently small so that:
\begin{equation}\label{actuallastcondition} N\varepsilon_n<16^{-n-2}, \end{equation}
\begin{equation}\label{squeezedcondition} \varepsilon_n<\eta(16^{-n-2})\textrm{, and } \end{equation}
\begin{equation}\label{lastlastcondition} \textrm{ if } f: U_n \rightarrow \{z : 1 < |z| < 1+\delta\} \textrm{ is conformal, then } \delta<1/100. \end{equation}

We will now recursively define a sequence of triangulations $\{\mathcal{T}_n\}_{n=0}^\infty$ of $\Chat$. First we introduce the following notation. Given a set $E\subset\mathbb{C}$, we denote the union of the unbounded components of $\Chat\setminus E$ by $\textrm{ex}(E)$, and by $\textrm{in}(E)$ the union of the bounded components of $\Chat\setminus E$. Let ${\mathcal T}_0$ be the triangulation obtained by applying Lemma \ref{first step} to $\varepsilon=\varepsilon_0$. We now describe how to define the triangulation ${\mathcal T}_{n}$, given ${\mathcal T}_{n-1}$. Our inductive hypothesis will be:
\begin{itemize} \item [$(\star)$] Each component of  
\begin{equation}\label{confgridimage}  \bigcup\left\{T\in\mathcal{T}_{n-1} : T\subset\textrm{in}(U_{n-1})\right\}\end{equation}
 is the conformal image of an equilateral grid polygon, and
  \item [$(\star\star)$] If $T\in\mathcal{T}_{n-1}$ satisfies $T\subset\textrm{in}(U_{n-1})$, then $\diam(T)<\varepsilon_{n-1}$.
  \end{itemize}
 where we note that $(\star)$, $(\star\star)$ hold true for $n=1$ by Lemma \ref{first step} and since (\ref{confgridimage}) lies inside $\mathbb{D}$ by (\ref{WLOG_assumption}).

Define $V_n$ to be the triangles in ${\mathcal T}_{n-1}$ that intersect $\textrm{ex}(U_n)$, so that
\begin{equation}\label{exprop} \textrm{ex}(U_n) \subset \bigcup_{T\in V_n}T. \end{equation}
Let $\mathcal{E}_n$ be a triangulation of $\mathbb{C}$ by Euclidean equilateral triangles of spherical diameter $<\delta_n$, where $\delta_n$ is sufficiently small so that
\begin{equation}\label{deltaperturbed} M\delta_n^\alpha<\varepsilon_{n+1}. \end{equation}
Denote by $W_n$ the union of triangles in $\mathcal{E}_n$ that intersect $\textrm{in}(U_n)$, so that
\begin{equation}\label{inprop} \textrm{in}(U_n) \subset \bigcup_{T\in W_n}T. \end{equation}


The region ``between'' $V_n$ and $W_n$ (or, more precisely, $\textrm{ex}(W_n)\cap\textrm{in}(V_n)$) consists of a union of topological annuli, one for each component of $\Gamma_n$ (see Figure \ref{Separate}). Let $A$ denote such an annulus. We claim that $A$ is a conformal grid annulus (see Definition \ref{confgridanndefn}), where we define the vertices on $\partial_oA$ as the vertices of the triangles $V_n$ lying on $\partial_o A$, and similarly the vertices on $\partial_iA$ are defined as the vertices of the triangles $W_n$ which lie on $\partial_iA$. Indeed, let $f_i$ be the identity mapping, and let $A_i$ be the union of triangles in $\mathcal{E}_n$ that are a subset of $A$ (with the inner boundary $\partial_iA_i$ coinciding with $\partial_iA$). Since $A_i$ is an equilateral grid annulus, we have shown the first half of Definition \ref{confgridanndefn} (namely conditions (1)-(3) for $k=i$).

\begin{figure}[!htbp]
\centerline{
  \includegraphics[height=2.5in]{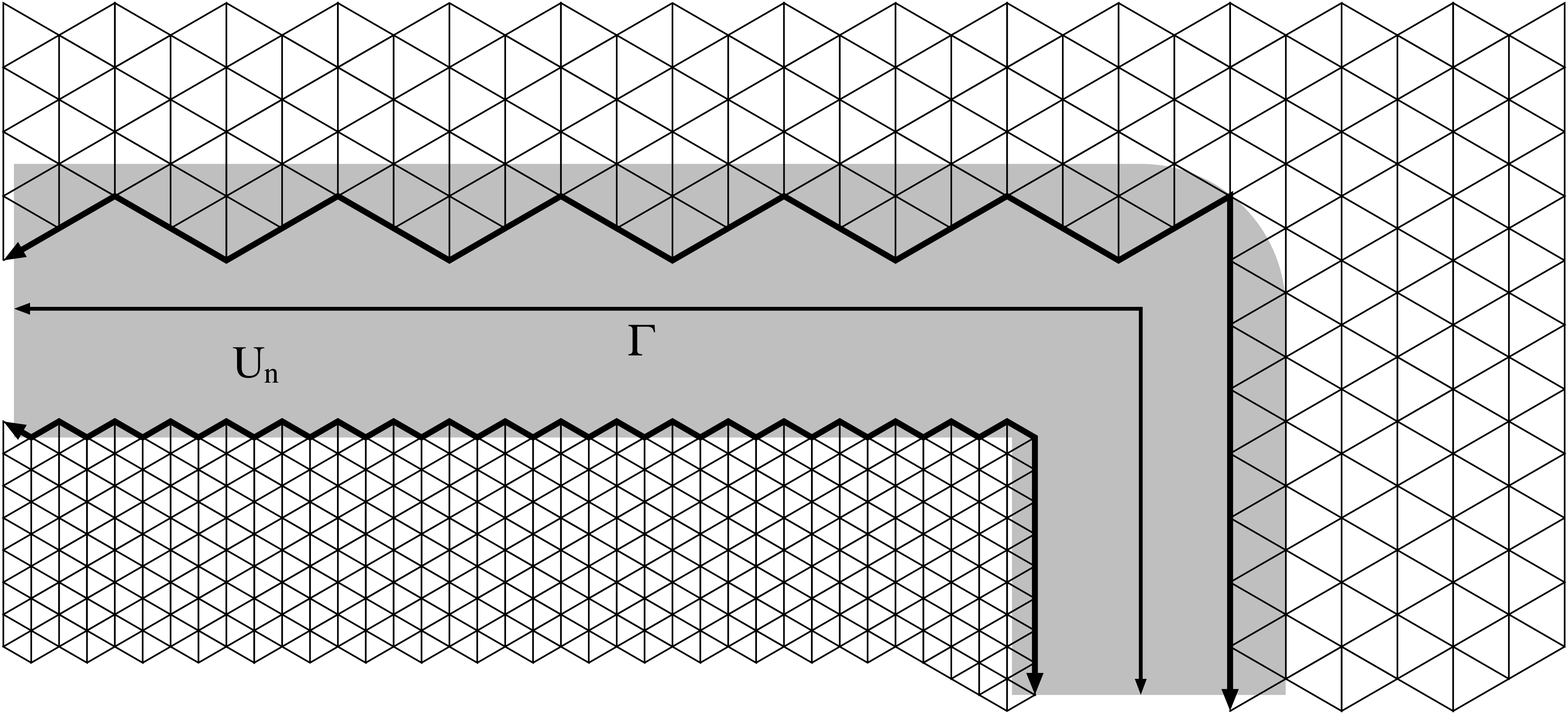}
}
\caption{ \label{Separate}
The shaded region represents part of  
$U_n$. The region $\textrm{in}(U_n)$ 
is covered 
by a Euclidean equilateral triangulation with 
side lengths $\epsilon_{n+1}$ (the small 
triangles). The region $\textrm{ex}(U_n)$ is covered
by elements of the triangulation ${\mathcal T}_{n-1}$
constructed at the previous stage (drawn here 
as larger equilateral triangles, but they need 
not be Euclidean triangles, only conformal 
images of such). 
The results of the previous section are used to 
triangulate the intervening region with the given 
boundary vertices, and then 
a quasiconformal correction will be applied to  
obtain an equilateral triangulation of the sphere.
}
\end{figure}

To finish verifying that $A$ is a conformal grid annulus, first note that by (\ref{exprop}) and (\ref{inprop}), we have $A\subset U_n$. By (\ref{Undefn}) and the inductive hypothesis $(\star\star)$, there is a topological annulus $\tilde{A}_o\subset A$ consisting of a union of triangles in $\mathcal{T}_{n-1}$ so that $\partial_o\tilde{A}_o=\partial_oA$, and \begin{equation}\label{innerboundaryinside} \partial_i\tilde{A}_o \subset \textrm{in}(\Gamma\cap A).\end{equation}
By Lemma \ref{surround}(4) and (\ref{Undefn}), (\ref{actuallastcondition}), we have that:
\begin{equation} U_n\subset\textrm{in}(U_{n-1}).\end{equation}
Thus, we conclude from the inductive hypothesis $(\star)$ that there is a conformal mapping $f_o: A\rightarrow f_o(A)$ so that $f_o$ and $A_o:=f_o(\tilde{A}_o)$ satisfy conditions (1)-(3) of the Definition \ref{confgridanndefn} of conformal grid annulus for $k=o$.

We have now proven that $A$ is a conformal grid annulus. By (\ref{lastlastcondition}), and our choice of $N$ together with Lemma \ref{thickness}, we have that both hypotheses of Lemma \ref{annuli triangulation} are satisfied. Hence, by Lemma \ref{annuli triangulation}, there is a triangulation $\tilde{\mathcal{T}}_n$ of $A$ (and of every other annular component of the region between $V_n$ and $W_n$) such that each triangle $T\in\tilde{\mathcal{T}}_n$ can be mapped to a Euclidean equilateral triangle by a $C$-quasiconformal map $\phi^T$ that multiplies arclength on each side of $T$. This induces a dilatation $\mu_T:=\phi^T_{\overline{z}}/\phi^T_z$ on each such triangle $T\in\tilde{\mathcal{T}}_n$. Extend $\tilde{\mathcal{T}}_n$ to a triangulation of $\Chat$ by adding the triangles in $V_n$ and $W_n$. 

We now define 
\begin{equation}\label{Tndefn} \mathcal{T}_n:=\phi_n(\tilde{\mathcal{T}}_n), \end{equation}
 where $\phi_n$ is a normalized quasiconformal solution to the Beltrami equation
 \begin{equation} \phi_{\overline{z}}=\mu\cdot\phi_z \end{equation}
  for $\mu$ defined a.e. in $\mathbb{C}$ as follows. Let
  \[ \mu := \begin{cases} 
      \mu_T & \textrm{ if } T\subset A, \\
      0 & \textrm{ if } T\in W_n, \\
      0 & \textrm{ if } T\in V_n \textrm{ and } T\cap\partial_oA=\emptyset. \\
   \end{cases}
\]
It remains to define $\mu$ on any triangle $T\in V_n$ intersecting $\partial_oA$. Note that $T\in\mathcal{T}_{n-1}$, and hence by the inductive hypothesis  $(\star)$, there is a conformal mapping $f$ of $T$ onto a Euclidean equilateral triangle $f(T)$. Also by the inductive hypothesis and our choice of $N$, the hypotheses of Lemma \ref{boundary triangles} are satisfied and hence there exists a $C$-quasiconformal mapping $g: T \rightarrow g(T)$ such that:
\begin{enumerate} \item $g(T)$ is a Euclidean equilateral triangle,
\item $g$ is length-multiplying on the edges of $T$ lying on $\partial A$, and
\item $g=f$ on the remaining edges of $T$. 
\end{enumerate}
Set $\mu=g_{\overline{z}}/g_z$ on $T$. This finishes the definition of $\mu$, and hence defines the triangulation (\ref{Tndefn}). The definition of $\mu$ was so as to ensure that for any adjacent triangles $T$, $T'$ in $\mathcal{T}_n$, there is an anti-conformal map $T\rightarrow T'$ satisfying Definition \ref{equil_triang_def} (see the proof of Lemma \ref{equilateral2} for a similar argument), so that $\mathcal{T}_n$ is an equilateral triangulation. Note furthermore that since $\mu=0$ in $W_n$, we have $\phi_n$ is conformal in $W_n$. Hence, since $W_n$ is an equilateral grid polygon, $(\star)$ holds with $n$ replacing $n-1$. Moreover, if $T\in\mathcal{T}_n$ satisfies $T\subset\textrm{in}(U_n)$, then $T$ is the image under $\phi_n$ of a triangle of diameter $\delta_n$, hence by (\ref{holderestimate}) and (\ref{deltaperturbed}) we have $(\star\star)$ also holds with $n$ replacing $n-1$. This concludes our recursive definition of the triangulations $(\mathcal{T}_n)_{n=0}^\infty$.


We now define the triangulation $\mathcal{T}_\infty$ satisfying the conclusion of Theorem B. Let $n\in\mathbb{N}$ and $T\in\mathcal{T}_n$ be a triangle so that $\phi_n^{-1}(T)\subset\textrm{ex}(U_n)$. Then 
\begin{equation}\label{seqoftri} \phi_{n+k}\circ...\circ\phi_{n+1}(T)\in\mathcal{T}_{n+k} \textrm{ for all } k\geq0. \end{equation}
Since the maps $(\phi_n)_{n=1}^\infty$ are uniformly Cauchy by (\ref{phismallpert}), the sequence (in $k$) of triangles (\ref{seqoftri}) converges to a triangle $T_\infty$ (with vertices/edges of $T_\infty$ defined as the limit of vertices/edges of (\ref{seqoftri})), and we define by $\mathcal{T}_\infty^n$ the collection of all such limit triangles. By our definition of $\mathcal{T}_n$, we have $\mathcal{T}_\infty^n\subset\mathcal{T}_\infty^{n+1}$. Define 
\begin{equation} \mathcal{T}_\infty:=\bigcup_{n=1}^\infty\mathcal{T}_\infty^n. \end{equation}
We claim that $\mathcal{T}_\infty$ is an equilateral triangulation of $D$ satisfying the conclusions of Theorem B.

First we show that
\begin{equation}\label{setequalityforfinaltriang}\bigcup_{T\in\mathcal{T}_\infty}T=D. \end{equation}
If $T\in\mathcal{T}_\infty$, there exists $n\in\mathbb{N}$ so that 
\begin{equation}\label{Tlimdefn} T=\lim_{k\rightarrow\infty}\phi_{n+k}\circ...\circ\phi_{n+1}(T')  \textrm{ for some } T' \in \mathcal{T}_n \textrm{ satisfying }  \phi_n^{-1}(T')\subset\textrm{ex}(U_n).\end{equation}
In particular, by Lemma \ref{surround}(3) and since $U_n$ surrounds $\Gamma_n$ we have that
\begin{equation} d(T', \partial D) > 16^{-n}. \end{equation}
That $T\subset D$ now follows from (\ref{phismallpert}) and (\ref{Tlimdefn}). On the other hand, if $z\in D$, then by Lemma \ref{surround} we have that
\begin{equation} z\in\bigcup_{T\in\mathcal{T}_n}T \textrm{ and } d\left(z,\partial\left(\bigcup_{T\in\mathcal{T}_n}T\right)\right)>2\cdot16^{-n}\end{equation}
for sufficiently large $n$. Thus, by (\ref{phismallpert}), we have that 
\begin{equation} z\in \phi_{n+k}\circ...\circ\phi_{n+1}\left( \bigcup_{T\in\mathcal{T}_n}T \right)\textrm{ for all } k, \end{equation}
and hence $z\in\bigcup_{T\in\mathcal{T}_\infty}T$. Thus we have proven (\ref{setequalityforfinaltriang}). 

That $\mathcal{T}_\infty$ is a triangulation follows from the definition of $\mathcal{T}_\infty$ as a limit of the triangulations $\mathcal{T}_n$, and so by (\ref{setequalityforfinaltriang}), $\mathcal{T}_\infty$ is a triangulation of $D$. In order to show that $\mathcal{T}_\infty$ is an equilateral triangulation, we need to show that there is an anti-conformal reflection between any two adjacent triangles $T$, $T' \in \mathcal{T}_\infty$ as in Definition \ref{equil_triang_def}. This follows since there are adjacent triangles $T_n$, $T_n'\in\mathcal{T}_n$ limiting on $T$, $T'$ (respectively), and the anti-conformal reflection mappings $T_n\rightarrow T_n'$ limit on the desired anti-conformal reflection $T\rightarrow T'$.

We now bound the degree of any vertex in $\mathcal{T}_\infty$. First note that for a vertex $v$ in $\mathcal{T}_n$, the degree of $v$ is either $6$ (this is the degree of any vertex in $\mathcal{E}_n$ or in $\mathcal{T}_0$), or else the degree of $v$ is bounded by the universal constant on the degree of any vertex arising by the application of Lemma \ref{annuli triangulation}. In particular, the degree of $v$ in $\mathcal{T}_n$ is bounded independently of $n$, $v$, $D$ and $\eta$. Hence the degree of any vertex $v$ in $\mathcal{T}_\infty$ is bounded independently of $v$, $D$ and $\eta$. 


Finally, we prove (\ref{eta bound}). Let $z\in D$. Fix $n$ so that
\begin{equation} z\in\textrm{in}(\Gamma_{n-1})\cap\textrm{ex}(\Gamma_n). \end{equation}
By Lemma \ref{surround}(3), we then have that:
\begin{equation} \dist(z,\partial D)\geq16^{-n}. \end{equation}
Hence, in order to prove (\ref{eta bound}), it suffices to show that any triangle $T\in\mathcal{T}_\infty$ containing $z$ satisfies
\begin{equation}\label{finalWTS} \diam(T)\leq\eta(16^{-n}). \end{equation}
Denote by $\mathcal{A}$ the collection of triangles in $\mathcal{T}_{n+1}$ that intersect
\begin{equation} \textrm{in}(U_{n-2})\cap\textrm{ex}(U_{n+1}). \end{equation}
By definition of $\mathcal{A}$, one of the following must hold for each triangle $T$ in $\mathcal{A}$:
\begin{enumerate}
\item $\phi_{n-2}^{-1}(T)\in\mathcal{E}_{n-2}$,
\item $T\subset U_{n-1}$,
\item $\phi_{n-1}^{-1}(T)\in\mathcal{E}_{n-1}$,
\item $T\subset U_n$, or
\item $\phi_{n}^{-1}(T)\in\mathcal{E}_{n}$.
\end{enumerate}
Let
\begin{equation} T_k:=\phi_{n+k}\circ...\circ\phi_{n+1}(T) \textrm{ for } T\in\mathcal{A}. \end{equation}
We claim that 
\begin{equation}\label{TKWTS} \diam(T_k)<\eta(16^{-n}) \textrm{ for all } k\geq1 \textrm{ and } T\in\mathcal{A}. \end{equation}
Indeed, by (\ref{deltaperturbed}), for $T$ as in cases (1), (3), (5) above we have that $\diam(T_k)$ is bounded by $\varepsilon_{n-2}$, $\varepsilon_{n-1}$, $\varepsilon_{n}$ (respectively). Hence (\ref{TKWTS}) follows for $T$ as in cases (1), (3), (5) by (\ref{squeezedcondition}). For $T$ as in cases (2), (4) we have by Lemma \ref{triangle size} that $\diameter(T)$ is bounded by $NC'\varepsilon_{n-2}$,  $NC'\varepsilon_{n-1}$ (respectively). Hence, applying (\ref{forthediamestimate}) to $\phi=\phi_{n+k}\circ...\circ\phi_{n+1}$ finishes the proof of (\ref{TKWTS}). Moreover, by Lemma \ref{surround}(4), the definition of $U_n$ and (\ref{phismallpert}), (\ref{actuallastcondition}), we have that $z\in\cup_{T\in\mathcal{A}}T_k$ for all $k$. Hence, by (\ref{TKWTS}), any triangle $T$ in $\mathcal{T}_\infty$ containing $z$ satisfies (\ref{finalWTS}), as needed.
\qed



\end{document}